\documentclass[a4paper,reqno,10pt]{amsart}
    \DeclareMathSizes{12}{12}{7}{6}
\usepackage{geometry}
\usepackage{amssymb,amsmath,mathrsfs,bm}
\usepackage{paralist}
\usepackage[usenames]{color}
\usepackage[all]{xy}
\usepackage{url}
\usepackage{braket}
\usepackage{graphicx}
\usepackage{transparent}
\usepackage{xcolor}
\usepackage[shortlabels]{enumitem}
\usepackage{mathtools}
\usepackage{tikz}
\usepackage{tikz-cd}

\usepackage{wasysym} % for pentagon symbol
\usepackage{stmaryrd} % \longmapsfrom
\makeatletter

\@addtoreset{equation}{section}
\makeatother
\usepackage{amsthm}
\usepackage{aliascnt}
\newtheorem{theorem}{Theorem}[section]
\newaliascnt{lemma}{theorem}
\newtheorem{lemma}[lemma]{Lemma}
\aliascntresetthe{lemma}

\newaliascnt{proposition}{theorem}
\newtheorem{proposition}[proposition]{Proposition}
\aliascntresetthe{proposition}

\newaliascnt{corollary}{theorem}
\newtheorem{corollary}[corollary]{Corollary}
\aliascntresetthe{corollary}

\theoremstyle{definition}
\newaliascnt{definition}{theorem}
\newtheorem{definition}[definition]{Definition}
\aliascntresetthe{definition}

\theoremstyle{definition}
\newaliascnt{remark}{theorem}
\newtheorem{remark}[remark]{Remark}
\aliascntresetthe{remark}

\newaliascnt{setup}{theorem}
\theoremstyle{definition}

\aliascntresetthe{setup} 

\newaliascnt{example}{theorem}
\newtheorem{example}[example]{Example}
\aliascntresetthe{example}

\newaliascnt{condition}{theorem}

\aliascntresetthe{condition}

\newaliascnt{construction}{theorem}

\aliascntresetthe{construction}

\newaliascnt{question}{theorem}
\newtheorem{question}[question]{Question}
\aliascntresetthe{question}

\newaliascnt{conjecture}{theorem}

\aliascntresetthe{conjecture}

\usepackage[hyperfootnotes=false]{hyperref}
\usepackage[capitalise,noabbrev]{cleveref}
\crefname{theorem}{Theorem}{Theorems}
\crefname{lemma}{Lemma}{Lemmas}
\crefname{proposition}{Proposition}{Propositions}
\crefname{corollary}{Corollary}{Corollaries}
\crefname{definition}{Definition}{Definitions}
\crefname{remark}{Remark}{Remarks}
\crefname{example}{Example}{Examples}
\crefname{condition}{Condition}{Conditions}
\crefname{construction}{Construction}{Constructions}
\crefname{claim}{Claim}{Claims}
\crefname{mainthm}{Theorem}{Theorems}
\crefname{maincor}{Corollary}{Corollaries}
\crefname{setup}{Setup}{Setups}

\DeclareMathOperator{\Mor}{Mor}

\DeclareMathOperator{\Hom}{Hom}
\DeclareMathOperator{\Mod}{\mathsf{Mod}}
\renewcommand{\mod}{\operatorname{\mathsf{mod}}}
\DeclareMathOperator{\defect}{\mathsf{def}}
\DeclareMathOperator{\coh}{\mathsf{coh}}
\DeclareMathOperator{\Ex}{\mathsf{Ex}}

\DeclareMathOperator{\add}{\mathsf{add}}

\newcommand{\op}{\mathsf{op}}

\renewcommand{\Im}{\operatorname{Im}}
\newcommand{\Cok}{\operatorname{Cok}}

\newcommand{\id}{\mathsf{id}}
\newcommand{\xto}{\xrightarrow}

\newcommand{\dg}{{\rm dg}}
\newcommand{\tr}{\mathsf{tr}}
\newcommand{\pretr}{\mathsf{pretr}}

\DeclareMathOperator{\cpx}{\mathrm{Cpx}}
\DeclareMathOperator{\Tot}{\mathrm{Tot}}
\DeclareMathOperator{\Fun}{\mathrm{Fun}}

\newcommand{\all}{\rm all}

\DeclareMathOperator{\Filt}{\mathrm{Filt}}

\newcommand{\lra}{\longrightarrow}
\newcommand{\dra}{\dashrightarrow}
\DeclareMathOperator{\cone}{\mathsf{Cone}}

\newcommand{\wtil}[1]{\widetilde{#1}}

	\newcommand{\deff}{\coloneqq}
	\newcommand{\sse}{\subseteq}

    \newcommand{\fs}{\mathfrak{s}}

	\newcommand{\BE}{\mathbb{E}}
	\newcommand{\BF}{\mathbb{F}}

	\newcommand{\CA}{\mathcal{A}}
	
	\newcommand{\CC}{\mathcal{C}}
	\newcommand{\CD}{\mathcal{D}}

	\newcommand{\CH}{\mathcal{H}}

	\newcommand{\CM}{\mathcal{M}}

	\newcommand{\CS}{\mathcal{S}}
	
	\newcommand{\CU}{\mathcal{U}}

    \newcommand{\A}{\mathscr{A}}
	\newcommand{\B}{\mathscr{B}}

    \newcommand{\T}{\mathscr{T}}

\usetikzlibrary{matrix,arrows,decorations.pathmorphing,positioning,decorations.pathreplacing}
\tikzset{commutative diagrams/.cd, 
mysymbol/.style = {start anchor=center, end anchor = center, draw = none}}
\tikzset{
labl/.style={anchor=north, rotate=90, inner sep=1mm}
}

\let\amph=& %needed for matrix environments in arrow labels
	
\tikzcdset{every label/.append style = {font = \footnotesize}}

\begin{document}
\setlength{\baselineskip}{15pt}
%%%%%%%%%%%%%%%%%%%%%%%%%%%%%%%%%%%%%%%%%%%%%%%%%%%%%%%%%%%%%%%%%
\title[Intrinsic classification of exact dg structures]
{Exact dg structures are classified intrinsically by bi-Serre subcategories}

\author[Ogawa]{Yasuaki Ogawa}
	\address{Faculty of Engineering Science, Kansai University, Suita-shi, Osaka 564-8680, Japan}
	\email{y\underline{ }ogawa@kansai-u.ac.jp} %



\begin{abstract}
The aim of this article is to leverage Enomoto's classification of exact structures to the dg level.
Let $\A$ be a connective additive idempotent complete dg category.
We establish an intrinsic Enomoto-type classification of exact dg structures on $\A$ in terms of bi-Serre subcategories in $\tr(\A)$.
Our proof is a direct dg argument based on dg totalizations of $3$-term complexes and dg
duality, rather than a reduction to the greatest exact dg structure and the classification of its substructures.
As a consequence, we establish a bijection between exact dg structures on $\A$ and bi-Serre subcategories in $\tr(\A)$.
Moreover, this bijection is an isomorphism of posets and yields a complete lattice structure on the class of exact dg structures on $\A$.
In particular, we provide another construction of Rump--Chen's greatest exact dg structure on $\A$.
\end{abstract}
\maketitle
\tableofcontents

%%%%%%%%%%%%%%%%%%%%%%%%%%%%%%%%%%%%%%%%%%%%%%%%%%%%%%%%%%%%%%%%%
%%%%%%%%%%%%%%%%%%%%%%%%%%%%%%%%%%%%%%%%%%%%%%%%%%%%%%%%%%%%%%%%%
\section*{Introduction}\label{sec:intro}
%%%%%%%%%%%%%%%%%%%%%%%%%%%%%%%%%%%%%%%%%%%%%%%%%%%%%%%%%%%%%%%%%
%%%%%%%%%%%%%%%%%%%%%%%%%%%%%%%%%%%%%%%%%%%%%%%%%%%%%%%%%%%%%%%%%
Quillen exact categories \cite{Qui73} and Grothendieck--Verdier triangulated categories \cite{Ver96} constitute fundamental frameworks for homological algebra.
They share the basic philosophy that both are designed to capture universal phenomena of a homological nature, and indeed exhibit many analogous structures.
On the other hand, their intersection consists essentially of semisimple categories, vividly illustrating the substantial conceptual gap between the two frameworks.
Nevertheless, Nakaoka and Palu introduced extriangulated categories, which encompass both exact and triangulated categories.
Since then, considerable effort has been devoted to unifying arguments that had previously been developed independently in the exact and triangulated settings.
Extriangulated categories are not merely a convenient device for such unification: they have also revealed structural phenomena that become visible in this broader framework, notably in the theories of cotorsion pairs \cite{LN19,HS20,MNO26} and Auslander--Reiten theory \cite{INP24}, and appear to be gaining acceptance as a foundational framework for homological algebra.
More recently, the dg enhancement of extriangulated categories introduced by Xiaofa Chen \cite{Che23,Che24a} has further advanced this line of research.

At this point, it is important to recall that some essential homological information cannot be captured by triangulated categories alone.
For example, higher $K$-theory of exact categories, introduced as a higher-dimensional extension of the Grothendieck group \cite{Qui73}, cannot be formulated as a functor on triangulated categories satisfying the expected properties \cite{Sch02}.
One solution is to pass to dg enhancements.
The standard dg enhancements of triangulated categories are provided by pretriangulated dg categories, and higher $K$-theory can be formulated with the desired properties once its domain is replaced by pretriangulated dg categories \cite{Wal85,TT90,Sch06}.
Since extriangulated categories include triangulated categories, they inevitably encounter the same difficulty.
From this perspective, Chen's introduction of a dg enhancement of extriangulated categories is particularly useful.
His definition is obtained by adapting Quillen's axioms for exact categories to the dg setting, and the term \emph{exact dg category} aptly reflects this construction.
If a dg category $\A$ is equipped with an exact dg structure $\CS$, then its homotopy category $H^0(\A)$ carries a canonical extriangulated structure.

The aim of this paper is to classify the exact dg structures that can be imposed on a given dg category $\A$.
Our starting point is Enomoto's classification of Quillen exact structures on an additive category $\CA$ \cite{Eno18}.
Let us describe his result in some detail.
Consider the full subcategory $\CC_2(\CA)$ of the functor category $\Mod\CA$ consisting of $\CA$-modules having both projective dimension and grade equal to $2$.
Enomoto's classification theorem establishes a bijection between exact structures on $\CA$ and bi-Serre subcategories of $\CC_2(\CA)$:
\begin{equation}\label{Enomoto_SM}
\Set{\text{exact structures on }\CA}
\underset{\mathsf{S}}{\overset{\mathsf{M}}{\rightleftarrows}}
\Set{\text{bi-Serre subcategories of }\CC_2(\CA)}.
\end{equation}
These assignments are explicit.
Given an exact category $(\CA,\CS)$, the assignment $\mathsf{M}$ associates with it the full subcategory consisting of all Auslander defect modules.
Conversely, the assignment $\mathsf{S}$ associates with a bi-Serre subcategory the $3$-term complexes in $\CA$ arising from projective resolutions of the modules belonging to that subcategory.

The main theorem of this paper is precisely a dg analogue of Enomoto's result and takes exactly the same form.

\begin{theorem}[{\cref{thm:correspondence}}]
\label{thm:A}
Let $\A$ be a connective additive dg category and assume that $\A$
is idempotent complete.
Then there exists a bijection between the following classes:
\begin{enumerate}[label=\textup{(\alph*)}]
\item
exact dg structures on $\A$;
\item
bi-Serre subcategories of $\CC_2(\A)\sse\tr(\A)$.
\end{enumerate}
\end{theorem}

A key technical ingredient in the proof is a dg-level refinement of the totalization construction. We construct totalization as a dg functor $\Tot_\A$ and establish its compatibility with the shifted dg duality $\Sigma^2(-)^\vee$:
\[
\begin{tikzcd}
\cpx^3_{\dg}(\A)^{\op}
\arrow{r}{(-)^{\op}}
\arrow{d}[swap]{\Tot_\A^{\op}}
&
\cpx^3_{\dg}(\A^{\op})
\arrow{d}{\Tot_{\A^{\op}}}
\\
\pretr(\A)^{\op}
\arrow{r}[swap]{\Sigma^2(-)^\vee}
&
\pretr(\A^{\op}),
\end{tikzcd}
\]
see \cref{lem:duality_on_defects}.
This compatibility makes the duality underlying the bi-Serre condition explicit at the dg level.
Here $\CC_2(\A)\sse\tr(\A)$ is the dg analogue of Enomoto's category $\CC_2(\CA)$, see \cref{def:C2}.
As a full subcategory of $\tr(\A)$, it is characterized by simple vanishing conditions.
It follows from \cref{thm:A} that the poset of bi-Serre subcategories forms a complete lattice and, consequently, that $\A$ admits a greatest exact dg structure.

Chen, on the other hand, proved the existence of the greatest exact dg structure directly by adapting Rump's method \cite{Rum19} to the dg setting.
Combining this result with the classification of extriangulated substructures established in \cite{Eno21} yields another classification theorem, which differs slightly from \cref{thm:A}.

\begin{theorem}[\cref{thm:correspondence_Chen}]
\label{thm:B}
Let $\A$ be a connective additive dg category.
Let $\CM_{\max}$ denote the full subcategory corresponding to the
greatest exact dg structure.
Then there exists a bijection between the following classes:
\begin{enumerate}[label=\textup{(\alph*)}]
\item
exact dg structures on $\A$;
\item[\textup{(b')}]
Serre subcategories of $\CM_{\max}$.
\end{enumerate}
\end{theorem}

The two classification theorems have complementary advantages.
Theorem~\ref{thm:B} describes exact dg structures in terms of ordinary Serre subcategories.
Its ambient category $\CM_{\max}$, however, is defined through the greatest exact dg structure and is therefore not specified a priori solely in terms of the underlying dg category $\A$.
By contrast, Theorem~\ref{thm:A} involves the apparently stronger condition of being bi-Serre, but its ambient category $\CC_2(\A)$ is defined intrinsically from $\A$ by explicit vanishing conditions in $\tr(\A)$, independently of any choice of an exact dg structure.
Comparing the two classifications also determines the position of $\CM_{\max}$ inside $\tr(\A)$: namely, $\CM_{\max}$ is the
greatest bi-Serre subcategory of $\CC_2(\A)$.

\medskip
\noindent
{\bf Notation and conventions.}
Throughout this paper, all dg categories are assumed to be small and additive, unless otherwise stated. Here a dg category $\A$ is called \emph{additive} if its homotopy category $H^0(\A)$ is additive.
We write $\CA=H^0(\A)$.

All dg modules and modules over additive categories are right modules.
For an additive category \(\CA\), we denote by $\Mod\CA$ and $\mod\CA$ the categories of right $\CA$-modules and finitely presented right $\CA$-modules, respectively.

A full subcategory of a category is always assumed to be closed under isomorphisms.
We say that a dg category $\A$ is \emph{idempotent complete} if $H^0(\A)$ is idempotent complete.

\medskip
\noindent
{\bf Use of AI.}
The author used OpenAI ChatGPT 5.6 Sol and Codex 5.6 Sol, both with medium reasoning effort, as auxiliary tools for locating relevant results in the literature, checking formal computations and sign conventions and polishing the English and \LaTeX{} presentation.
All mathematical arguments and citations were independently verified and finalized by the author.
The author takes full responsibility for the accuracy, originality, and integrity of the paper.

%%%%%%%%%%%%%%%%%%%%%%%%%%%%%%%%%%%%%%%%%%%%%%%%%%%%%%%%%%%%%%%%%
\section{On exact dg categories}\label{sec:exact_dg}
%%%%%%%%%%%%%%%%%%%%%%%%%%%%%%%%%%%%%%%%%%%%%%%%%%%%%%%%%%%%%%%%%
%%%%%%%%%%%%%%%%%%%%%%%%%%%%%%%%%%%%%%%%%%%%%%%%%%%%%%%%%%%%%%%%%
\subsection{Definition and basic properties}\label{subsec:basics_of_exact_dg}
%%%%%%%%%%%%%%%%%%%%%%%%%%%%%%%%%%%%%%%%%%%%%%%%%%%%%%%%%%%%%%%%%
We first recall the definition of an exact dg category from
\cite{Che23,Che24a}.
This notion is a dg analogue of Quillen's exact categories and includes pretriangulated dg categories as special cases.
Unless otherwise stated, all dg categories are assumed to be \emph{additive}.

\begin{definition}\label{def:connective_dg_cat}
A dg category $\A$ is said to be \emph{connective} if
$H^i\A(X,Y)=0$ for all objects $X,Y\in\A$ and all $i>0$.
For an arbitrary dg category $\A$, its \emph{connective cover}
$\tau_{\leq0}\A$ is the dg category with the same objects as $\A$ and with morphism complexes $(\tau_{\leq0}\A)(X,Y)=\tau_{\leq0}\bigl(\A(X,Y)\bigr)$.
Note that the canonical dg functor $\tau_{\leq0}\A\to\A$ induces an isomorphism $H^0(\tau_{\leq0}\A)\xrightarrow{\sim}H^0\A$.
\end{definition}

\begin{example}\label{ex:connective_cover}
Every additive category, regarded as a dg category concentrated in degree zero, is connective.
More importantly for our purposes, if $\T'$ is a pretriangulated dg category, then its connective cover $\tau_{\leq0}\T'$ is a connective dg enhancement of the triangulated category $H^0\T'$.
\end{example}

A \emph{$3$-term complex} $\xi$ is a sextuple $(A,B,C,f,g,h)$ defined as a diagram in $\A$ of the form below
\begin{equation}\label{diag:3-term_complex}
\begin{tikzcd}[row sep=0.6cm]
A \arrow{r}{f}\arrow[bend right]{rr}[swap]{h}& B\arrow{r}{g} & C,
\end{tikzcd}
\end{equation}
where $|f|=|g|=0, |h|=-1$ and $d(f)=0, d(g)=0$ and $d(h)=-gf$.
To save space, the diagram \eqref{diag:3-term_complex} is often abbreviated as 
\[
\xi=(A\xto{f}B\xto{g}C,h).
\]
We denote by $\cpx^3(\A)$ the category whose objects are the
$3$-term complexes in $\A$.
Let
\[
\xi'=(A'\xto{f'}B'\xto{g'}C',h')
\]
be another $3$-term complex.
A morphism from $\xi$ to $\xi'$ in $\cpx^3(\A)$ is represented by
a homotopy class of sextuples
$
\Xi=(r_0,r_1,r_2,s_1,s_2,t)
$
illustrated by the following diagram:
\begin{equation}\label{diag:morph_of_3-term_complex}
\begin{tikzcd}[row sep=0.8cm]
A
 \arrow{r}{f}
 \arrow[bend left]{rr}{h}
 \arrow{d}[swap]{r_0}
 \arrow[draw=red]{rd}[swap]{\textcolor{red}{s_1}}
 \arrow[draw=blue]{rrd}{\textcolor{blue}{t}}
&B
 \arrow{r}{g}
 \arrow{d}[swap]{r_1}
 \arrow[draw=red]{rd}{\textcolor{red}{s_2}}
&C\arrow{d}{r_2}
\\
A'
 \arrow{r}[swap]{f'}
 \arrow[bend right]{rr}[swap]{h'}
&B'\arrow{r}[swap]{g'}
&C'.
\end{tikzcd}
\end{equation}
Here, $|r_i|=0, d(r_i)=0$ for $i=0,1,2$,
\[
\begin{aligned}
|s_1|=-1,
&&d(s_1)=f'r_0-r_1f,\\
|s_2|=-1,
&&d(s_2)=g'r_1-r_2g,
\end{aligned}
\]
and
\[
t\colon A\to C',
\qquad |t|=-2,
\]
satisfies
$
d(t)
=
r_2\circ h-h'\circ r_0-s_2\circ f-g'\circ s_1
$.
We refer to \cite[\S 3.2]{Che24a} for the homotopy relation and
the composition of morphisms in $\cpx^3(\A)$, see also \S \ref{subsec:totalization_dg}.

\begin{lemma}\label{lem:isomorphisms_in_cpx3}
\cite[Proposition~3.15]{Che24a}
Let
$
\Xi\colon\xi\to\xi'
$
be a morphism in $\cpx^3(\A)$ represented by the diagram \eqref{diag:morph_of_3-term_complex}. Then $\Xi$ is an isomorphism if and
only if $r_0,r_1,r_2$ are isomorphisms in $H^0\A$.
\end{lemma}

\begin{remark}
A $3$-term complex as in \eqref{diag:3-term_complex} is called a \emph{$3$-term homotopy complex} in \cite{Che23,Che24a}.
Throughout this paper, we omit the prefix ``homotopy'' for simplicity and use the shorter terms, such as $3$-term complex, kernel, and left exact, whenever no confusion is likely.
Moreover, our category $\cpx^3(\A)$ is denoted by
$\CH_{3t}(\A)$ in \cite[Definition~3.14]{Che24a}.
\end{remark}

\begin{remark}\label{rem:defining_diagram_of_totalization}
By a slight abuse of notation, we identify each object $X\in\A$
with the representable dg $\A$-module $\A(-,X)$ in $\CC_{\dg}(\A)$ via the Yoneda embedding $\A\to\CC_{\dg}(\A)$.
A $3$-term complex \eqref{diag:3-term_complex} gives rise to the following commutative diagram in $\CC_{\dg}(\A)$:
\begin{equation}\label{diag:defect}
\begin{tikzcd}[row sep=0.8cm]
A\arrow{r}{f}\arrow{d}[swap]{\begin{bsmallmatrix}
-f\\
-h
\end{bsmallmatrix}}
&B\arrow[equal]{d}{}\arrow{r}{\begin{bsmallmatrix}
0\\
1
\end{bsmallmatrix}}
&U_\xi \arrow{r}{\begin{bsmallmatrix}
1\amph 0
\end{bsmallmatrix}}\arrow{d}[swap]{\begin{bsmallmatrix}
-h\amph g
\end{bsmallmatrix}}\arrow[draw=red]{rd}{\textcolor{red}{s}}
&\Sigma A \arrow{d}{}
\\
V_\xi \arrow{r}[swap]{\begin{bsmallmatrix}
-1\amph 0
\end{bsmallmatrix}}
&B \arrow{r}[swap]{g}
&C \arrow{r}[swap]{\begin{bsmallmatrix}
0\\
1
\end{bsmallmatrix}}
\arrow{d}[swap]{\begin{bsmallmatrix}
0\\
0\\
1
\end{bsmallmatrix}}
&\Sigma V_\xi \arrow{d}{}
\\
{}
&{}
&M_{\xi} \arrow[equal]{r}{}\arrow{d}{}
&M_{\xi} \arrow{d}{}
\\
{}
&{}
&\Sigma U_\xi \arrow{r}
&\Sigma^2 A
\end{tikzcd}
\end{equation}
where we set $
U_\xi=\cone(f), V_\xi=\Sigma^{-1}\cone(g)$ and $M_\xi=\cone(\begin{bmatrix}
-h\amph g
\end{bmatrix})$.
Besides, the red arrow is
\[
s=
\begin{bmatrix}
0&-1\\
0&0
\end{bmatrix}.
\]
The two upper rows and the two rightmost columns of
\eqref{diag:defect} give rise to distinguished triangles in $\CD(\A)$.
For later use, we also put
\[
\varphi_\xi=
\begin{bmatrix}
-f\\
-h
\end{bmatrix}
\colon A\to V_\xi
\quad\text{and}\quad
\psi_\xi=
\begin{bmatrix}
-h\amph g
\end{bmatrix}
\colon U_\xi\to C .
\]
More precisely, by abuse of notation, $h$ also denotes the
corresponding degree-zero morphism $A\to C[-1]$.
\end{remark}

\begin{definition}\label{def:short_exact_sequence}
A $3$-term complex $\xi$ is \emph{left exact} if, for every
$X\in\A$ and every $i\leq0$, the morphism $\varphi_\xi$ induces an
isomorphism
\[
\Hom_{\CD(\A)}(X,\Sigma^iA)
\xrightarrow{\sim}
\Hom_{\CD(\A)}(X,\Sigma^iV_\xi).
\]
It is \emph{right exact} if, for every $X\in\A$ and every $i\leq0$,
the morphism $\psi_\xi$ induces an isomorphism
\[
\Hom_{\CD(\A)}(C,\Sigma^iX)
\xrightarrow{\sim}
\Hom_{\CD(\A)}(U_\xi,\Sigma^iX).
\]
A $3$-term complex which is both left and right exact is called a
\emph{short exact sequence}.
\end{definition}

We conclude this subsection by recalling Chen's definition of an exact dg category.
Throughout the article, we use $\CS\sse\cpx^3(\A)$ to denote an isomorphism-closed class of $3$-term complexes in $\A$.

\begin{definition}\label{def:exact_dg}
\cite[Def.~4.1]{Che24a}
Let $\A$ be an additive dg category.
An \emph{exact structure} on $\A$ is an isomorphism-closed class $\CS$ of objects in $\cpx^3(\A)$ consisting of short exact $3$-term complexes.
A $3$-term complex
\[
\xi=(X\xto{f}Y\xto{g}Z,h)
\]
of $\CS$ is called a \emph{conflation}, while $f$ and $g$ are called an \emph{inflation} and a \emph{deflation}, respectively.
The class $\CS$ is required to satisfy the following conditions:
\begin{enumerate}[label=\textup{(Ex\arabic*)},start=0,leftmargin=45pt]
\item\label{EX0}
$\id_A$ is a deflation for all $A\in\A$.
\item\label{EX1}
Deflations are closed under composition.
\item\label{EX2}
For any deflation $g\colon Y\to Z$ and any morphism $u\colon Z'\to Z$ in $Z^0\A$, a homotopy pullback of $g$ along $u$ exists, and the induced morphism $g'\colon Y'\to Z'$ is a deflation.
\end{enumerate}
\begin{enumerate}[
    label=\textup{(Ex2$^{\op}$)},
    ref=\textup{(Ex2$^{\op}$)},
    leftmargin=45pt
]
\item\label{EX2op}
The dual of \ref{EX2} holds.
\end{enumerate}
The pair $(\A,\CS)$ is called an \emph{exact dg category}.
\end{definition}

\begin{example}\label{ex:exact_category_as_exact_dg}
\cite[Example~4.6]{Che24a}
Let $\A$ be an additive category, regarded as a dg category
concentrated in degree zero.
Then a $3$-term complex
\[
A\xto{f}B\xto{g}C
\]
is short exact in the dg sense if and only if it is a kernel-cokernel
pair in $\A$.
Moreover, exact dg structures on $\A$ are precisely Quillen exact
structures on $\A$.
Thus, exact dg categories concentrated in degree zero are identified
with exact categories.
\end{example}

The following fundamental result shows that an exact dg category naturally provides a dg enhancement of an extriangulated category.

\begin{theorem}\label{thm:exact_dg_to_extri}
\cite[Theorem~4.30]{Che24a}
Let $(\A,\CS)$ be an exact dg category.
Then $H^0(\A)$ admits a canonical extriangulated structure
\[
\bigl(H^0(\A),\BE_\CS,\fs_\CS\bigr)
\]
in the sense of \cite{NP19}.
\end{theorem}

%%%%%%%%%%%%%%%%%%%%%%%%%%%%%%%%%%%%%%%%%%%%%%%%%%%%%%%%%%%%%%%%%
\subsection{Exact dg substructures and extriangulated substructures}
\label{subsec:relative_theory}
%%%%%%%%%%%%%%%%%%%%%%%%%%%%%%%%%%%%%%%%%%%%%%%%%%%%%%%%%%%%%%%%%
Following \cite[\S 4.4]{Che24a}, we recall the correspondence between
exact dg substructures of an exact dg category $(\A,\CS)$ and
extriangulated substructures of the associated extriangulated category
$
\bigl(H^0(\A),\BE_\CS,\fs_\CS\bigr)
$.

Let $(\CC,\BE,\fs)$ be an extriangulated category.
Recall from \cite[\S 5.1]{INP24} that an additive subbifunctor $\BF\sse\BE$ is said to be \emph{closed} if it satisfies any of the following equivalent conditions:
\begin{itemize}
\item 
$(\CC,\BF,\fs|_{\BF})$ is an extriangulated category;
\item 
$\fs|_\BF$-inflations are closed under composition; and
\item 
$\fs|_\BF$-deflations are closed under composition.
\end{itemize}
In this case, we call $(\BF,\fs|_{\BF})$ an \emph{extriangulated substructure} of $(\BE,\fs)$.
This notion provides an extriangulated analogue of the relative exact structures considered in \cite{AS93,DRSSK99}.

We next recall the notion of a contravariant defect.

\begin{definition}\label{def:defect_subcategory}
\cite[Definition~2.4]{Oga21}; see also
\cite[Definition~2.8]{Eno21}.
Let $(\CC,\BE,\fs)$ be a skeletally small extriangulated category.
\begin{enumerate}
\item 
Let $A\overset{f}{\lra}B\overset{g}{\lra} C\overset{\delta}{\dra}$ be an $\fs$-triangle and define $\wtil{\delta}$ to be the cokernel of $\CC(-,g) : \CC(-,B)\to\CC(-,C)$ in $\Mod\CC$.
We call $\wtil{\delta}$ the \emph{(contravariant) defect} of the $\fs$-triangle.
\item
We denote by $\defect\BE$ the full subcategory of $\Mod\CC$ consisting of $\CC$-modules which are isomorphic to defects of some $\fs$-triangle.
We call $\defect\BE$ the \emph{defect subcategory} of $\Mod\CC$ with respect to the extriangulated structure $(\BE,\fs)$.
\end{enumerate}
\end{definition}

For a skeletally small additive category $\CC$, a right $\CC$-module $M$ is said to be \emph{coherent} if $M$ is finitely presented and every finitely generated submodule of $M$ is finitely presented.
We denote by $\coh\CC$ the full subcategory of $\Mod\CC$ consisting of coherent right $\CC$-modules.
By \cite[Proposition~2.7]{Eno21}, $\coh\CC$ is a wide subcategory of $\Mod\CC$; in particular, it is an exact abelian subcategory.

\begin{proposition}\label{prop:defect_subcat_is_Serre}\cite[Proposition~2.9]{Eno21}
Let $(\CC,\BE,\fs)$ be a skeletally small extriangulated category.
Then $\defect\BE$ is a Serre subcategory of $\coh\CC$.
\end{proposition}

The class of closed additive subbifunctors of $\BE$ forms a poset under inclusion. Moreover, arbitrary intersections of closed subbifunctors are again closed by \cite[Corollary~3.14]{HLN21}.
Similarly, the exact dg substructures of $(\A,\CS)$ form a poset under inclusion.
The following correspondence theorem is due to Chen and builds on Enomoto's classification of closed subbifunctors; see \cite[Theorem~B]{Eno21}.

\begin{proposition}\label{prop:bijection_between_substructures}
\cite[Theorem~4.37]{Che24a}
Let $(\A,\CS)$ be an exact dg category.
Then the following three posets are isomorphic:
\begin{enumerate}
\item
the poset of exact dg substructures of $(\A,\CS)$;
\item
the poset of extriangulated substructures of
$\bigl(H^0(\A),\BE_\CS,\fs_\CS\bigr)$;
\item
the poset of Serre subcategories of $\defect\BE_\CS$.
\end{enumerate}
\end{proposition}

%%%%%%%%%%%%%%%%%%%%%%%%%%%%%%%%%%%%%%%%%%%%%%%%%%%%%%%%%%%%%%%%%
\subsection{Totalizations and defects}
\label{subsec:totalizations_and_defects}
%%%%%%%%%%%%%%%%%%%%%%%%%%%%%%%%%%%%%%%%%%%%%%%%%%%%%%%%%%%%%%%%%
We now return to the dg setting and introduce the totalization of a $3$-term complex. For a short exact $3$-term complex, its totalization provides a dg lift of the contravariant defect introduced in the preceding subsection.

Throughout this subsection, let $\A$ be a connective dg category.
Let $\xi=(A\xto{f}B\xto{g}C,h)$ be a $3$-term complex in $\A$.
Recall from Remark~\ref{rem:defining_diagram_of_totalization} that $\xi$ determines an object
\[
M_\xi=\cone(\psi_\xi)\in\pretr(\A)
\]
appearing in the diagram \eqref{diag:defect}.
We call $M_\xi$ the \emph{totalization} of $\xi$ and write $M_\xi=\Tot(\xi)$.
As explained below, if $\xi$ is a short exact sequence, $M_\xi$ recovers the
contravariant defect of $\xi$ in the sense of Auslander. 

\begin{lemma}\label{lem:totalization_recovers_defect}
Let $(\A,\CS)$ be an exact category, regarded as an exact dg category
concentrated in degree zero.
For a conflation
$
\xi=(A\xto{f}B\xto{g}C)
$
in $\CS$, let $\delta$ denote the corresponding extension.
Then there is a natural isomorphism
\[
M_\xi\cong\widetilde{\delta}
\]
in $\CD(\A)$, where the contravariant defect
$\widetilde{\delta}$ is regarded as a dg $\A$-module concentrated in
degree zero.
\end{lemma}
\begin{proof}
The conflation $\xi$ induces an exact sequence
\[
0\longrightarrow\A(-,A)
\longrightarrow\A(-,B)
\longrightarrow\A(-,C)
\longrightarrow\widetilde{\delta}
\longrightarrow0
\]
in $\Mod\A$.
By construction, $M_\xi$ is the total complex of the first three
representable modules. Hence the above projective resolution induces
the asserted isomorphism in $\CD(\A)$.
\end{proof}

This motivates the following terminology.

\begin{definition}\label{def:defect}
If $\xi$ is a short exact sequence, we call $M_\xi$ the
\emph{contravariant defect}, or simply the \emph{defect} of $\xi$.
For any class $\CS$ of short exact sequences in $\A$, we denote by $\CM_\CS$ the full subcategory of $\tr(\A)$ consisting of all objects isomorphic to $M_\xi$ for some $\xi\in\CS$.
\end{definition}

\begin{remark}
Chen calls an object in the heart $\Mod H^0(\A)$ of $\CD(\A)$ obtained as the totalization of a conflation a \emph{defective object};
see \cite[Def.~3.8]{Che24b}.
We use the terminology \emph{defect} instead, in accordance with Auslander's
terminology for exact and extriangulated categories; see \cite[Chapter~III, \S4]{Aus78} and
\cite[\S IV.4]{ARS}.
\end{remark}

The following lemma shows that the exactness of $\xi$ can be detected entirely from the position of $M_\xi$ inside $\tr(\A)$.
More precisely, left and right exactness correspond to vanishing conditions against representable dg modules on the two sides.

\begin{lemma}\label{lem:defect_vanishing}
Let $\A$ be a connective dg category, let $\xi$ be a $3$-term
complex in $\A$, and let $M_\xi\in\tr(\A)$ be its totalization.
Then the following statements hold.
\begin{enumerate}
\item
The $3$-term complex $\xi$ is left exact if and only if
\[
\Hom_{\CD(\A)}(X,\Sigma^iM_\xi)=0
\]
for every $X\in\A$ and every $i\neq0$.

\item
The $3$-term complex $\xi$ is right exact if and only if
\[
\Hom_{\CD(\A)}(M_\xi,\Sigma^iX)=0
\]
for every $X\in\A$ and every $i\neq2$.
\end{enumerate}
Consequently, $\xi$ is a short exact sequence if and only if
\begin{equation}\label{diag:defect_vanishing}
\begin{cases}
\Hom_{\CD(\A)}(X,\Sigma^iM_\xi)=0
    & (i\neq0),\\
\Hom_{\CD(\A)}(M_\xi,\Sigma^iX)=0
    & (i\neq2)
\end{cases}
\end{equation}
for every $X\in\A$.
\end{lemma}
\begin{proof}
According to the diagram \eqref{diag:defect}, we put $U_\xi=\cone(f)$ and
$V_\xi=\Sigma^{-1}\cone(g)$.
We also denote by
\[
\varphi_\xi=
\begin{bmatrix}
-f\\
-h
\end{bmatrix}
\colon A\longrightarrow V_\xi,
\qquad
\psi_\xi=
\begin{bmatrix}
-h\amph g
\end{bmatrix}
\colon U_\xi\longrightarrow C
\]
the morphisms induced by $\xi$.
The third and rightmost columns of
\eqref{diag:defect} yield triangles in $\CD(\A)$:
\begin{align}
&U_\xi
\xrightarrow{\psi_\xi}C\longrightarrow
M_\xi\longrightarrow\Sigma U_\xi
\qquad\text{and}
\label{tri:totalization_right}
\\
&A
\xrightarrow{\varphi_\xi}V_\xi\longrightarrow
\Sigma^{-1}M_\xi\longrightarrow\Sigma A.
\label{tri:totalization_left}
\end{align}

(1)
By definition, $\xi$ is left exact if and only if $\varphi_\xi$ induces
isomorphisms
\begin{equation}\label{isom:from_left_exact}
\Hom_{\CD(\A)}(X,\Sigma^iA)
\xrightarrow{\sim}
\Hom_{\CD(\A)}(X,\Sigma^iV_\xi)
\end{equation}
for every $X\in\A$ and every $i\leq0$.
Applying $\Hom_{\CD(\A)}(X,-)$ to
\eqref{tri:totalization_left}, we have the following long exact sequence:
\begin{equation}\label{seq:long_exact_left}
\Hom_{\CD(\A)}(X,\Sigma^iA)
\overset{\varphi_\xi\circ -}{\lra}
\Hom_{\CD(\A)}(X,\Sigma^iV_\xi)
\lra
\Hom_{\CD(\A)}(X,\Sigma^{i-1}M_\xi)
\lra
\Hom_{\CD(\A)}(X,\Sigma^{i+1} A).
\end{equation}
The isomorphisms \eqref{isom:from_left_exact}, together with connectivity, imply
\[
\Hom_{\CD(\A)}(X,\Sigma^iM_\xi)=0
\qquad(i<0).
\]
Moreover, the second row of \eqref{diag:defect} gives $\Hom_{\CD(\A)}(X,\Sigma^{i+1}V_\xi)=0$ for every $i>0$.
The long exact sequence \eqref{seq:long_exact_left} therefore yields $\Hom_{\CD(\A)}(X,\Sigma^iM_\xi)=0$ for every $i>0$.

Conversely, we suppose that
$
\Hom_{\CD(\A)}(X,\Sigma^iM_\xi)=0
$
for any $i\neq 0$.
Then the leftmost morphisms in \eqref{seq:long_exact_left} are isomorphisms for $i\leq 0$, which shows $\xi$ is left exact.

(2)
Dually, $\xi$ is right exact if and only if $\psi_\xi$ induces
isomorphisms
\begin{equation}\label{isom:from_right_exact}
\Hom_{\CD(\A)}(C,\Sigma^iX)
\xrightarrow{\sim}
\Hom_{\CD(\A)}(U_\xi,\Sigma^iX)
\end{equation}
for every $X\in\A$ and every $i\leq0$. Applying $\Hom_{\CD(\A)}(-,X)$ to
\eqref{tri:totalization_right}, we obtain the following long exact sequence:
\[
\Hom_{\CD(\A)}(M_\xi,\Sigma^iX)
\longrightarrow
\Hom_{\CD(\A)}(C,\Sigma^iX)
\longrightarrow
\Hom_{\CD(\A)}(U_\xi,\Sigma^iX)
\longrightarrow
\Hom_{\CD(\A)}(M_\xi,\Sigma^{i+1}X).
\]
If $
\Hom_{\CD(\A)}(M_\xi,\Sigma^iX)=0$ for any $i\neq2$,
then we have \eqref{isom:from_right_exact}.
Hence $\xi$ is right exact.

Conversely, if $\xi$ is right exact, then the isomorphisms \eqref{isom:from_right_exact} first give
$
\Hom_{\CD(\A)}(M_\xi,\Sigma^iX)=0
$ for any $i\leq 0$.
By connectivity, we have $\Hom_{\CD(\A)}(C,\Sigma X)=0$ and $\Hom_{\CD(\A)}(M_\xi,\Sigma X)=0$.
Moreover, the first row of \eqref{diag:defect}  gives $\Hom_{\CD(\A)}(U_\xi,\Sigma^iX)=0$ for every $i>1$.
The long exact sequence associated with
\eqref{tri:totalization_right} therefore yields
\[
\Hom_{\CD(\A)}(M_\xi,\Sigma^iX)=0
\qquad(i>2).
\]
Thus only the degree $i=2$ may be nonzero.
\end{proof}

\begin{remark}
Mochizuki and Nakaoka introduce higher $n$-exact sequences by means of vanishing conditions analogous to those in \eqref{diag:defect_vanishing}; see \cite[Definition~3.1]{MN26}.
For $n=1$, their notion agrees with that of a short exact $3$-term complex used here; see \cite[Remark~3.5]{MN26}.
Thus, \cref{lem:defect_vanishing} may also be viewed as the $n=1$ case of their higher framework.
\end{remark}

These conditions motivate the definition of the category
$\CC_2(\A)$ below.

\begin{definition}\label{def:C2}
Let $\CC_2(\A)$ be the full subcategory of $\tr(\A)$ consisting of the objects $M$ satisfying the vanishing condition \eqref{diag:defect_vanishing}:
\[
\CC_2(\A)\deff\Set{M\in\tr(\A)|
\begin{aligned}
&\Hom_{\CD(\A)}(X,\Sigma^i M)=0\   (i\neq 0)\\
&\Hom_{\CD(\A)}(M,\Sigma^i X)=0\   (i\neq 2)
\end{aligned}
\text{\quad for\ any\ } X\in\A
}.
\]
\end{definition}

Let $\CS_{\all}$ be the class of all short exact sequences in $\A$.
By \cref{lem:defect_vanishing}, for every $3$-term complex
$\xi\in\cpx^3(\A)$, we have
\begin{equation}\label{diag:short_exact_C2}
\xi\in\CS_{\all}
\quad\Longleftrightarrow\quad
M_\xi\in\CC_2(\A).
\end{equation}
Thus, we have $\CM_\CS\sse\CC_2(\A)$ for any subclass $\CS\sse\CS_{\all}$.

For an object $M\in\CC_2(\A)$, we construct a $3$-term presentation of $M$, namely, a $3$-term complex $\xi$ of the form \eqref{diag:3-term_complex} such that $M\cong M_\xi$ in $\tr(\A)$.
The main tool is the bounded weight structure on $\tr(\A)$ induced by the connectivity of $\A$.

For a full subcategory $\CC$ of a triangulated category, we denote
by $\add(\CC)$ its closure under finite direct sums and retracts.
For full subcategories $\CC_1,\CC_2$, we denote by
$\CC_1*\CC_2$ the full subcategory consisting of objects $X$
which admit a triangle
\[
C_1\to X\to C_2\to\Sigma C_1
\]
with $C_i\in\CC_i$.

\begin{lemma}\label{lem:weight_structure}
The subcategory $H^0(\A)$ generates a bounded weight structure $(\tr(\A)^{\geq 0},\tr(\A)^{\leq 0})$ on $\tr(\A)$.
More explicitly, we have
\begin{align*}
\tr(\A)^{\geq 0}
&=
\add\left(
\bigcup_{i\geq0}
\Sigma^{-i}H^0(\A)*\Sigma^{-i+1}H^0(\A)*\cdots *H^0(\A)
\right)
\quad\text{and}\quad
\\
\tr(\A)^{\leq 0}
&=
\add\left(
\bigcup_{i\geq0}
H^0(\A)*\Sigma H^0(\A)*\cdots *\Sigma^iH^0(\A)
\right).
\end{align*}
The associated weight heart is $\add(H^0(\A))$. In particular, if $H^0(\A)$ is idempotent complete, then the weight heart is $H^0(\A)$.
\end{lemma}
\begin{proof}
By connectivity, $H^0(\A)$ is a negative subcategory of $\tr(\A)$; namely $\Hom_{\tr(\A)}(A,\Sigma^iB)=0$ for any $i>0$ and any $A,B\in H^0(\A)$.
Since $H^0(\A)$ generates $\tr(\A)$, the assertion follows from
\cite[Theorem~4.3.2(II)]{Bon10}.
\end{proof}

We say that $\A$ is \emph{idempotent complete} if $H^0(\A)$ is idempotent complete.
In the remainder of this subsection, we assume that $\A$ is idempotent complete.

\begin{lemma}\label{lem:weight_amplitude}
If $M\in\CC_2(\A)$, then $M$ has weights in $[-2,0]$.
More precisely, there exists a $3$-term complex $\xi$ in $\A$ of the form
\begin{equation}\label{diag:weight_amplitude}
\begin{tikzcd}[row sep=0.6cm]
A \arrow{r}{f}\arrow[bend right]{rr}[swap]{h}
  & B\arrow{r}{g}
  & C
\end{tikzcd}
\end{equation}
such that $M_\xi\cong M$ in $\tr(\A)$. In particular, $M\in
H^0(\A)*\Sigma H^0(\A)*\Sigma^2H^0(\A)$.
\end{lemma}
\begin{proof}
By the definition of $\CC_2(\A)$, we see
\[
\Hom_{\tr(\A)}(X,\Sigma^iM)=0\ (i>0)\quad \text{and}\quad \Hom_{\tr(\A)}(M,\Sigma^iX)=0\ (i>2)
\]
for every $X\in\A$.
Recall the following orthogonal descriptions:
\begin{align*}
\tr(\A)^{\leq0}
&=
\Set{N |
\Hom_{\tr(\A)}(X,\Sigma^iN)=0
\text{ for all }X\in\A\text{ and }i>0},
\\
\tr(\A)^{\geq0}
&=
\Set{N |
\Hom_{\tr(\A)}(N,\Sigma^iX)=0
\text{ for all }X\in\A\text{ and }i>0}.
\end{align*}
It follows that $M\in\tr(\A)^{\leq0}\cap\tr(\A)^{\geq-2}$.

By \cite[Proposition~1.5.6]{Bon10}, we can choose a weight Postnikov tower of $M$ compatible with this
weight range:
We obtain triangles
\[
C\longrightarrow M\longrightarrow M_1
\xrightarrow{u}\Sigma C
\]
and
\[
\Sigma B\longrightarrow M_1\longrightarrow\Sigma^2A
\xrightarrow{-\Sigma^2\bar f}\Sigma^2B
\]
for some $A,B,C\in H^0(\A)$ and some
$\overline{f}\in H^0(\A)(A,B)$.
Choose a closed degree-zero representative $f\in Z^0\A(A,B)$
of $\overline{f}$. The second triangle identifies $M_1$ with
$\Sigma\cone(f)$. Hence we obtain
a triangle
\[
\cone(f)\xrightarrow{\overline{q}}C
\longrightarrow M\longrightarrow\Sigma\cone(f).
\]
Since $\cone(f)=\Sigma A\oplus B$ as an underlying graded
$\A$-module, we can write $q=
\begin{bsmallmatrix}
-h\amph g
\end{bsmallmatrix}
$
for some $g\in Z^0(\A)(B,C)$ and $h\in\A^{-1}(A,C)$.
With our sign convention, the condition $dq=0$ is precisely
$d(h)=-gf$.
Thus, we obtain a desired $3$-term complex $\xi$ of the form
\eqref{diag:weight_amplitude} such that $M_\xi=\cone(q)\cong M$ in $\tr(\A)$.
Finally, the two triangles above show that
$
M\in
H^0(\A)*\Sigma H^0(\A)*\Sigma^2H^0(\A)
$.
\end{proof}

Combining \cref{lem:weight_amplitude} with the vanishing characterization of short exact sequences, we obtain the key representation property of $\CC_2(\A)$:
every object of $\CC_2(\A)$ is the totalization of a short exact sequence.

\begin{corollary}\label{cor:represents_obj_of_C2}
For an object $M\in\tr(\A)$, the following conditions are equivalent:
\begin{enumerate}
\item
$M\in\CC_2(\A)$.
\item
There exists a short exact sequence $\xi$ in $\A$ such that
$
M\cong M_\xi
$
in $\tr(\A)$.
\end{enumerate}
\end{corollary}
\begin{proof}
Suppose that $M\in\CC_2(\A)$. By
\cref{lem:weight_amplitude}, there exists a $3$-term complex $\xi$
in $\A$ such that $M\cong M_\xi$ in $\tr(\A)$. Hence $M_\xi\in\CC_2(\A)$, and \cref{lem:defect_vanishing} implies that $\xi$ is a short exact sequence.

Conversely, suppose that $M\cong M_\xi$ for some short exact
sequence $\xi$ in $\A$. By \cref{lem:defect_vanishing},
we have $M_\xi\in\CC_2(\A)$, and hence $M\in\CC_2(\A)$.
\end{proof}

We rephrase \cref{cor:represents_obj_of_C2} in terms of defect subcategories.

\begin{corollary}\label{cor:C2_as_all_defects}
We have $\CC_2(\A)=\CM_{\CS_{\all}}$.
\end{corollary}

%%%%%%%%%%%%%%%%%%%%%%%%%%%%%%%%%%%%%%%%%%%%%%%%%%%%%%%%%%%%%%%%%
\subsection{Totalization as a dg functor}\label{subsec:totalization_dg}
%%%%%%%%%%%%%%%%%%%%%%%%%%%%%%%%%%%%%%%%%%%%%%%%%%%%%%%%%%%%%%%%%
We now lift the totalization construction to the dg level.
Let $\T=\CC_{\dg}(\A)$ be the dg category of right dg
$\A$-modules.
The cone construction then defines a dg functor
\[
\cone\colon\Mor(\T)\longrightarrow\T;
\]
see \cite[Proposition~6.1.6]{CC19} and cf.~\cite[\S 2.9]{Dri04}.
For completeness, we briefly recall the \emph{dg morphism category}
$\Mor(\T)$. Its objects are closed morphisms
$$
f\colon X\longrightarrow Y
$$
of degree zero in $\T$. For two objects $f\colon X\to Y$ and
$f'\colon X'\to Y'$, the degree-$p$ component of the morphism complex
is
\[
\Mor(\T)(f,f')^p
=
\Set{
\begin{pmatrix}
\alpha&0\\
s&\beta
\end{pmatrix}
|
\alpha\in\T(X,X')^p,
\beta\in\T(Y,Y')^p,
s\in\T(X,Y')^{p-1}
}.
\]
Its differential is given by
\[
d
\begin{pmatrix}
\alpha&0\\
s&\beta
\end{pmatrix}
=
\begin{pmatrix}
-d(\alpha)&0\\
d(s)+f'\alpha-(-1)^p\beta f&d(\beta)
\end{pmatrix},
\]
and composition is given by matrix multiplication:
\[
\begin{pmatrix}
\alpha'&0\\
s'&\beta'
\end{pmatrix}
\circ
\begin{pmatrix}
\alpha&0\\
s&\beta
\end{pmatrix}
=
\begin{pmatrix}
\alpha'\alpha&0\\
s'\alpha+\beta's&\beta'\beta
\end{pmatrix}.
\]
In particular, a closed morphism of degree zero consists of closed
morphisms $\alpha$ and $\beta$ together with a homotopy $s$ satisfying
$
d(s)=\beta f-f'\alpha.
$
Thus, objects of $\Mor(\Mor(\T))$ may be regarded as
homotopy-commutative squares in $\T$:
\[
\begin{tikzcd}
X \arrow{r}{f}\arrow{d}[swap]{\alpha}
  \arrow[red]{rd}{s}
& Y\arrow{d}{\beta}\\
X'\arrow{r}[swap]{f'}
& Y'
\end{tikzcd}
\]

Following \cite[Def.~3.14]{Che24a}, let $\B$ be the dg path category of the following graded quiver with relations
\begin{equation}
0\overset{a}{\lra} 1\overset{b}{\lra} 2
\end{equation}
where $|a| = |b| = 0, d(a) = 0, d(b) = 0$ and $ba = 0$.
We put
\[
\cpx^3_{\dg}(\T)
\deff
\Fun_{A_\infty}(\B,\T),
\]
the dg category of strictly unital $A_\infty$-functors from $\B$ to
$\T$.
Its objects are identified with $3$-term complexes $\xi=(A\xto{f}B\xto{g}C,h)$ in $\T$.
At the level of objects, such a $3$-term complex may be identified
with the following object of $\Mor(\Mor(\T))$:
\begin{equation}\label{diag:homotopy_square_from_xi}
\begin{tikzcd}
A \arrow{r}{f}\arrow{d}[swap]{0}
\arrow[red]{rd}{-h}
& B\arrow{d}{g}\\
0\arrow{r}[swap]{0}
& C
\end{tikzcd}
\end{equation}
Indeed, the equality $d(-h)=gf$ is precisely the homotopy-commutativity
condition for this square.

\begin{lemma}
The above assignment extends to a fully faithful dg functor
$
\iota\colon
\cpx^3_{\dg}(\T)\longrightarrow\Mor(\Mor(\T)).
$
\end{lemma}
\begin{proof}
The assignment $\iota$ has already been defined on objects.
It remains to define it on morphisms.
Let $\Xi\colon\xi\to\xi'$ be a homogeneous morphism of degree $n$ in $\cpx^3_{\dg}(\T)$.
The morphism
$
\Xi=(r_0,r_1,r_2,s_1,s_2,t)
$ is displayed as in the diagram \eqref{diag:morph_of_3-term_complex}:
\[
\begin{tikzcd}[row sep=0.8cm]
A
 \arrow{r}{f}
 \arrow[bend left]{rr}{h}
 \arrow{d}[swap]{r_0}
 \arrow[draw=red]{rd}[swap]{\textcolor{red}{s_1}}
 \arrow[draw=blue]{rrd}{\textcolor{blue}{t}}
&B
 \arrow{r}{g}
 \arrow{d}[swap]{r_1}
 \arrow[draw=red]{rd}{\textcolor{red}{s_2}}
&C\arrow{d}{r_2}
\\
A'
 \arrow{r}[swap]{f'}
 \arrow[bend right]{rr}[swap]{h'}
&B'\arrow{r}[swap]{g'}
&C'.
\end{tikzcd}
\]
where
\[
|r_0|=|r_1|=|r_2|=n,\qquad
|s_1|=|s_2|=n-1,\qquad
|t|=n-2.
\]
We first consider a morphism of $3$-term complexes, namely, we suppose $|r_i|=0, d(r_i)=0$ for $i=0,1,2$ and
\[
\begin{aligned}
&|s_1|=-1,
&&d(s_1)=f'r_0-r_1f,\\
&|s_2|=-1,
&&d(s_2)=g'r_1-r_2g,\\
&|t|=-2,
&&d(t)=r_2\circ h-h'\circ r_0-s_2\circ f-g'\circ s_1 .
\end{aligned}
\]
Then we define a closed degree-zero morphism
$\begin{psmallmatrix}
\varphi & 0\\
\sigma & \psi
\end{psmallmatrix}\colon \iota(\xi)\to\iota(\xi')$
in $\Mor(\Mor(\T))$,
where
\[
\varphi=\begin{pmatrix}
r_0 & 0\\
-s_1 & r_1
\end{pmatrix},\quad
\sigma=\begin{pmatrix}
0 & 0\\
-t & -s_2
\end{pmatrix}, \quad
\psi=\begin{pmatrix}
0 & 0\\
0 & r_2
\end{pmatrix}.
\]
Actually, we can verify that $\iota(\Xi)$ is closed by the following calculations:
The diagrams \eqref{diag:homotopy_square_from_xi} which represent $\iota(\xi)$ and $\iota(\xi')$ give morphisms 
\[
\alpha=\begin{pmatrix}
    0 & 0\\
    -h & g
\end{pmatrix}
\colon f\to 0
\quad\text{and}\quad
\alpha'=\begin{pmatrix}
    0 & 0\\
    -h' & g'
\end{pmatrix}
\colon f'\to 0
\]
in $Z^0\Mor(\T)$.
Since $\varphi$ and $\psi$ are closed morphisms, we have 
\begin{equation*}  
d\begin{pmatrix}
\varphi & 0\\
\sigma & \psi
\end{pmatrix}
=
\begin{pmatrix}
-d\varphi & 0\\
d\sigma+\alpha'\varphi-\psi\alpha & d\psi
\end{pmatrix}
=
\begin{pmatrix}
0 & 0\\
d\sigma+\alpha'\varphi-\psi\alpha & 0
\end{pmatrix} .
\end{equation*}
It remains to check the lower left entry is zero:
\begin{align*}
d\sigma+\alpha'\varphi-\psi\alpha
&=
\begin{pmatrix}
0 & 0\\
-dt-s_2f & -ds_2
\end{pmatrix}+
\begin{pmatrix}
0 & 0\\
-h'r_0-g's_1 & g'r_1
\end{pmatrix}-
\begin{pmatrix}
0 & 0\\
-r_2h & r_2g
\end{pmatrix}
=
\begin{pmatrix}
0 & 0\\
0 & 0
\end{pmatrix}
\end{align*}
Now let us return to the general case, that is,
\[
\Xi\colon\xi\longrightarrow\xi'
\]
is a homogeneous morphism of degree $n$. We define
\[
\iota(\Xi)
=
\begin{pmatrix}
\varphi_\Xi&0\\
\sigma_\Xi&\psi_\Xi
\end{pmatrix}
\quad \text{by}\quad
\varphi_\Xi=
\begin{pmatrix}
r_0&0\\
(-1)^{n+1}s_1&(-1)^nr_1
\end{pmatrix},
\ 
\sigma_\Xi=
\begin{pmatrix}
0&0\\
-t&-s_2
\end{pmatrix},
\ 
\psi_\Xi=
\begin{pmatrix}
0&0\\
0&r_2
\end{pmatrix}.
\]
Here, $\varphi_\Xi$ and $\psi_\Xi$ have degree $n$, while
$\sigma_\Xi$ has degree $n-1$. For $n=0$, this agrees with the
construction above.
A direct calculation using the differentials and compositions in the
two morphism dg categories shows that
\[
d\bigl(\iota(\Xi)\bigr)=\iota(d\Xi)
\quad
\text{and}
\quad
\iota(\Theta\circ\Xi)=\iota(\Theta)\circ\iota(\Xi)
\]
for homogeneous composable morphisms $\Xi$ and $\Theta$.
It is also clear from the definition that $\iota$ preserves identities.
Hence, $\iota$ is a dg functor.
Since $0$ is a strict zero object, every homogeneous morphism between two objects in the image of $\iota$ is uniquely of the form described above.
Thus, $\iota$ induces an isomorphism on each morphism complex and is
fully faithful.
\end{proof}

Let us denote by
$
\mathbf{Y}_\A\colon\A\to\CC_{\dg}(\A)
$
the Yoneda dg functor.
Postcomposition with $\mathbf{Y}_\A$ induces a dg functor
\[
(\mathbf{Y}_\A)_*
\colon
\cpx^3_{\dg}(\A)
\longrightarrow
\cpx^3_{\dg}\bigl(\CC_{\dg}(\A)\bigr).
\]
Now let us define the dg totalization functor as follows:
\[
\Tot_\A\colon \cpx^3_{\dg}(\A)\xto{(\mathbf{Y}_\A)_*}\cpx^3_{\dg}(\T)\xto{\iota}\Mor(\Mor(\T))\xto{\Mor(\cone)} \Mor(\T)\xto{\cone}\T .
\]
Since its image is contained in $\pretr(\A)$, the dg functor $\Tot_\A$ factors through the
dg functor
\begin{equation}
\label{diag:dg-totalization}
\Tot_\A\colon\cpx^3_{\dg}(\A)\longrightarrow\pretr(\A).
\end{equation}

\begin{remark}\label{rem:dg-totalization}
\begin{enumerate}
\item 
The dg functor $\Tot_\A$ is not necessarily fully faithful.
\item 
Chen gives an alternative construction of the dg totalization functor in terms of twisted complexes; see \cite[\S 3.30]{Che23}.
\end{enumerate}
\end{remark}

\begin{corollary}\label{cor:totalization_functor}
The dg totalization functor \eqref{diag:dg-totalization} induces a
functor
\[
\Tot_\A\colon
\cpx^3(\A)\longrightarrow\tr(\A),
\qquad
\xi\longmapsto M_\xi.
\]
In particular, if $\xi\cong\eta$ in $\cpx^3(\A)$, then
$M_\xi\cong M_\eta$ in $\tr(\A)$.
\end{corollary}
\begin{proof}
Applying $H^0$ to the dg functor $\Tot_\A$ gives a functor
$\cpx^3(\A)\to \tr(\A)$.
Moreover, the construction of $\Tot_\A$
identifies $\Tot_\A(\xi)$ with the totalization $M_\xi$ defined in
\cref{rem:defining_diagram_of_totalization}.
\end{proof}

%%%%%%%%%%%%%%%%%%%%%%%%%%%%%%%%%%%%%%%%%%%%%%%%%%%%%%%%%%%%%%%%%
\subsection{The dg duality}\label{subsec:dg_dual}
%%%%%%%%%%%%%%%%%%%%%%%%%%%%%%%%%%%%%%%%%%%%%%%%%%%%%%%%%%%%%%%%%
Having characterized $\CC_2(\A)$ in terms of defects, we next compare it with $\CC_2(\A^{\op})$ via dg duality. To this end, we recall the dg $\A$-dual.
The assignment
\[
(-)^\vee\colon
\pretr(\A)^{\op}\xto{\sim}\pretr(\A^{\op}),
\qquad
M\longmapsto
M^\vee\deff\pretr(\A)(M,-)|_{\A},
\]
defines a dg $\A$-duality. Thus, for $A\in\A$, we have
$
M^\vee(A)=\pretr(\A)(M,A).
$
We also consider the shifted dg $\A$-duality
\[
\Sigma^2(-)^\vee\colon
\pretr(\A)^{\op}\xto{\sim}\pretr(\A^{\op}).
\]
In particular, $\Sigma^2M^\vee\cong\pretr(\A)(\Sigma^{-2}M,-)|_{\A}$.

\begin{remark}
When $\A=A$ is a dg algebra, the above duality is the standard one
$
M\longmapsto\Hom_A^\bullet(M,A)
$, see \cite[\S 3.1, (3.6)]{Shk13}.
\end{remark}

\begin{lemma}\label{lem:shifted_duality_on_C2}
The shifted dg $\A$-duality induces a duality
\[
\Sigma^2(-)^\vee\colon
\CC_2(\A)^{\op}\xto{\sim}\CC_2(\A^{\op}).
\]
\end{lemma}
\begin{proof}
Let $M\in\CC_2(\A)$. For every $X\in\A$ and $i\in\mathbb Z$,
the dg $\A$-duality yields natural isomorphisms
\[
\Hom_{\CD(\A^{\op})}
 \bigl(X,\Sigma^i(\Sigma^2M^\vee)\bigr)
\cong
\Hom_{\CD(\A)}
 \bigl(M,\Sigma^{i+2}X\bigr)
\]
and
\[
\Hom_{\CD(\A^{\op})}
 \bigl(\Sigma^2M^\vee,\Sigma^iX\bigr)
\cong
\Hom_{\CD(\A)}
 \bigl(X,\Sigma^{i-2}M\bigr).
\]
The first group vanishes for $i\neq0$, while the second vanishes
for $i\neq2$. Hence, we have $\Sigma^2M^\vee\in\CC_2(\A^{\op})$.
The corresponding construction for $\A^{\op}$ is a quasi-inverse,
since $\Sigma^2\bigl(\Sigma^2M^\vee\bigr)^\vee\cong M$.
Thus, the shifted dg duality restricts to the asserted duality.
\end{proof}

\begin{lemma}\label{lem:duality_on_3-term_complexes}
Taking opposites defines a dg equivalence
\[
(-)^{\op}\colon
\cpx^3_{\dg}(\A)^{\op}
\xto{\sim}
\cpx^3_{\dg}(\A^{\op}).
\]
\end{lemma}
\begin{proof}
We first define the assignment on objects.
Let
\[
\xi=(A\xto{f}B\xto{g}C,h)
\]
be a $3$-term complex in $\A$.
We define $\xi^{\op}$ to be the following diagram in $\A^{\op}$:
\[
\begin{tikzcd}[row sep=0.6cm]
C
\arrow{r}{g^{\op}}
\arrow[bend right]{rr}[swap]{h^{\op}}
&
B\arrow{r}{f^{\op}}
&
A.
\end{tikzcd}
\]
Since
\[
d(h^{\op})
=(dh)^{\op}
=(-g\circ f)^{\op}
=-f^{\op}\circ g^{\op},
\]
the diagram $\xi^{\op}$ is a $3$-term complex in $\A^{\op}$.

Let
\[
\Xi=(r_0,r_1,r_2,s_1,s_2,t)\colon\xi\longrightarrow\xi'
\]
be a homogeneous morphism of degree $n$ in $\cpx^3_{\dg}(\A)$, as displayed in \eqref{diag:morph_of_3-term_complex},
where
\[
|r_0|=|r_1|=|r_2|=n,\qquad
|s_1|=|s_2|=n-1,\qquad
|t|=n-2.
\]
Writing
\[
d_{\cpx}(\Xi)
=
(\widehat{r_0},\widehat{r_1},\widehat{r_2},
 \widehat{s_1},\widehat{s_2},\widehat{t}),
\]
the differential in $\cpx^3_{\dg}(\A)$ is given by
\[
\begin{aligned}
\widehat{r_i}
&=d(r_i),\\
\widehat{s_1}
&=
d(s_1)+(-1)^{n+1}f'r_0+(-1)^nr_1f,\\
\widehat{s_2}
&=
d(s_2)+(-1)^{n+1}g'r_1+(-1)^nr_2g,\\
\widehat{t}
&=
d(t)
+(-1)^ng's_1
+h'r_0
+(-1)^ns_2f
+(-1)^{n+1}r_2h.
\end{aligned}
\]
We define
\[
\Xi^{\op}=(r_2^{\op},r_1^{\op},r_0^{\op},-s_2^{\op},-s_1^{\op},-t^{\op})\colon(\xi')^{\op}\longrightarrow\xi^{\op}
\]
by the following diagram:
\begin{equation}\label{diag:opposite_morphism_of_3-term_complex}
\begin{tikzcd}[row sep=1.0cm, column sep=1.5cm]
C'
 \arrow{r}{(g')^{\op}}
 \arrow[bend left]{rr}{(h')^{\op}}
 \arrow{d}[swap]{r_2^{\op}}
 \arrow[draw=red]{rd}[swap]{\textcolor{red}{-s_2^{\op}}}
 \arrow[draw=blue]{rrd}{\textcolor{blue}{-t^{\op}}}
&
B'
 \arrow{r}{(f')^{\op}}
 \arrow{d}[swap]{r_1^{\op}}
 \arrow[draw=red]{rd}{\textcolor{red}{-s_1^{\op}}}
&
A'\arrow{d}{r_0^{\op}}
\\
C
 \arrow{r}[swap]{g^{\op}}
 \arrow[bend right]{rr}[swap]{h^{\op}}
&
B\arrow{r}[swap]{f^{\op}}
&
A
\end{tikzcd}
\end{equation}
It can be checked by a direct calculation that the assignment on each morphism complex is a chain map, namely, $d_{\cpx}(\Xi^{\op})=(d_{\cpx}(\Xi))^\op$.
Indeed, the differential
\[
d_{\cpx}(\Xi^{\op})=(\widehat{r_2^{\op}},\widehat{r_1^{\op}},\widehat{r_0^{\op}},
 \widehat{-s_2^{\op}},\widehat{-s_1^{\op}},\widehat{-t^{\op}})
\]
is given by
\[
\begin{aligned}
\widehat{r_i^{\op}}
&=d(r_i^{\op})=d(r_i)^{\op},\\
\widehat{-s_2^{\op}}
&=
d(-s_2^{\op})+(-1)^{n+1}g^{\op}r_2^{\op}+(-1)^nr_1^{\op}(g')^{\op}\\
&=-d(s_2)^{\op}-(-1)^{n+1}(g'r_1)^{\op}-(-1)^{n}(r_2g)^{\op}=-\widehat{s_2}^{\op},\\
\widehat{-s_1^{\op}}
&=
d(-s_1^{\op})+(-1)^{n+1}f^{\op}r_1^{\op}+(-1)^nr_0^{\op}(f')^{\op}\\
&=-d(s_1)^{\op}-(-1)^{n+1}(f'r_0)^{\op}-(-1)^{n}(r_1f)^{\op}=-\widehat{s_1}^{\op},\\
\widehat{-t^{\op}}
&=
d(-t^{\op})
+(-1)^{n+1}f^{\op}s_2^{\op}
+h^{\op}r_2^{\op}
+(-1)^{n+1}s_1^{\op}(g')^{\op}
+(-1)^{n+1}r_0^{\op}(h')^{\op}\\
&=
-d(t)^{\op}
-(-1)^{n}(g's_1)^{\op}
-(h'r_0)^{\op}
-(-1)^{n}(s_2f)^{\op}
-(-1)^{n+1}(r_2h)^{\op}
=-\widehat{t}^{\op}.
\end{aligned}
\]

It remains to verify compatibility with compositions.
Let
\[
\Xi=(r_0,r_1,r_2,s_1,s_2,t)\colon\xi\longrightarrow\xi'
\qquad\text{and}\qquad
\Theta=
(r'_0,r'_1,r'_2,s'_1,s'_2,t')\colon\xi'\longrightarrow\xi''
\]
be homogeneous morphisms of degrees $n$ and $m$, respectively.
Since the source is the opposite dg category, we need to verify that
\[
\Xi^{\op}\circ\Theta^{\op}
=
(-1)^{mn}(\Theta\circ\Xi)^{\op}.
\]
By the definition of composition in $\cpx^3_{\dg}(\A)$, we have
\[
\Theta\circ\Xi
=
\left(
r'_0r_0,\,
r'_1r_1,\,
r'_2r_2,\,
r'_1s_1+(-1)^ns'_1r_0,\,
r'_2s_2+(-1)^ns'_2r_1,\,
r'_2t+t'r_0+(-1)^{n+1}s'_2s_1
\right) .
\]
Moreover, $(\Theta\circ\Xi)^{\op}$ is given by the sextuple
\begin{align*}
&(-1)^{mn}r_2^{\op}(r'_2)^{\op},\\
&(-1)^{mn}r_1^{\op}(r'_1)^{\op},\\
&(-1)^{mn}r_0^{\op}(r'_0)^{\op},\\
&-(-1)^{mn-m}s_2^{\op}(r'_2)^{\op}-(-1)^{mn}r_1^{\op}(s'_2)^{\op}=(-1)^{mn}\bigl(
-r_1^{\op}(s'_2)^{\op}-(-1)^ms_2^{\op}(r'_2)^{\op}
\bigr) ,\\
&-(-1)^{mn-m}s_1^{\op}(r'_1)^{\op}-(-1)^{mn}r_0^{\op}(s'_1)^{\op}=(-1)^{mn}\bigl(
-r_0^{\op}(s'_1)^{\op}-(-1)^m s_1^{\op}(r'_1)^{\op}
\bigr),\\
&-(-1)^{mn-2m}t^{\op}(r'_2)^{\op}-(-1)^{mn-2n}r_0^{\op}(t')^{\op}-(-1)^{n+1}(-1)^{(m-1)(n-1)}s_1^{\op}(s'_2)^{\op}\\
&\quad=(-1)^{mn}\bigl(
-r_0^{\op}(t')^{\op}-t^{\op}(r'_2)^{\op}+(-1)^{m+1}s_1^{\op}(s'_2)^{\op}
\bigr) .
\end{align*}
On the other hand, the opposite morphisms are
\[
\Xi^{\op}=(r_2^{\op},r_1^{\op},r_0^{\op},-s_2^{\op},-s_1^{\op},-t^{\op})
\quad\text{and}\quad
\Theta^{\op}=\bigl((r'_2)^{\op},(r'_1)^{\op},(r'_0)^{\op},-(s'_2)^{\op},-(s'_1)^{\op},-(t')^{\op}\bigr) .
\]
Their composite $\Xi^{\op}\circ \Theta^{\op}$ is given by the sextuple
\begin{align*}
&r_2^{\op}(r'_2)^{\op}, \\
&r_1^{\op}(r'_1)^{\op}, \\
&r_0^{\op}(r'_0)^{\op}, \\
&-r_1^{\op}(s'_2)^{\op}-(-1)^m s_2^{\op}(r'_2)^{\op}, \\
&-r_0^{\op}(s'_1)^{\op}-(-1)^m s_1^{\op}(r'_1)^{\op}, \\
&-r_0^{\op}(t')^{\op}-t^{\op}(r'_2)^{\op}+(-1)^{m+1}s_1^{\op}(s'_2)^{\op} .
\end{align*}
Comparing the above two sextuples, we obtain
$
\Xi^{\op}\circ\Theta^{\op}
=
(-1)^{mn}(\Theta\circ\Xi)^{\op}$.
Thus the assignment is compatible with compositions. Together with
the preceding calculation concerning the differentials, this shows
that $(-)^{\op}$ defines a dg functor. Applying the same construction
twice gives the identity dg functor. Hence, $(-)^{\op}$ is a dg
equivalence.
\end{proof}

\begin{lemma}\label{lem:duality_on_defects}
There is a natural isomorphism of dg functors
\[
\Tot_{\A^{\op}}\circ(-)^{\op}
\cong
\Sigma^2(-)^\vee\circ\Tot_\A^{\op}.
\]
Equivalently, the following diagram commutes up to a natural
isomorphism:
\[
\begin{tikzcd}
\cpx^3_{\dg}(\A)^{\op}
\arrow{r}{(-)^{\op}}
\arrow{d}[swap]{\Tot_\A^{\op}}
&
\cpx^3_{\dg}(\A^{\op})
\arrow{d}{\Tot_{\A^{\op}}}
\\
\pretr(\A)^{\op}
\arrow{r}[swap]{\Sigma^2(-)^\vee}
&
\pretr(\A^{\op}).
\end{tikzcd}
\]
In particular, for every $3$-term complex $\xi$ in $\A$, there is a
natural isomorphism $M_{\xi^{\op}}\cong\Sigma^2(M_\xi)^\vee$.
\end{lemma}
\begin{proof}
Let
\[
\xi=(A\xto{f}B\xto{g}C,h)
\]
be a $3$-term complex in $\A$.
By the explicit cone construction of the totalization in
\eqref{diag:defect}, we have
\[
\cone(f)=\left(
\Sigma A\oplus B, \begin{bmatrix}
    d_{\Sigma A} & 0 \\
    f & d_{B}
\end{bmatrix}
\right)
\quad
\text{and}
\quad
M_\xi=\left(
\Sigma^2A\oplus \Sigma B\oplus C,
\begin{bmatrix}
    d_{\Sigma^2A} & 0 & 0\\
    -f & d_{\Sigma B} & 0\\
    -h & g & d_C
\end{bmatrix}
\right),
\]
where, by abuse of notation, $f$, $g$, and $h$ also denote the corresponding shifted morphisms, all of which have degree $1$.
Using the canonical identifications
$
(\Sigma^iX)^\vee\cong\Sigma^{-i}X^\vee
$
and identifying $X^\vee$ with the corresponding object $X$ of
$\A^{\op}$,
we obtain
\[
\Sigma^2(M_\xi)^\vee=\left(
A\oplus \Sigma B\oplus \Sigma^2C,
\begin{bmatrix}
    d_{A} & -f^{\op} & -h^{\op}\\
    0 & d_{\Sigma B} & g^{\op}\\
    0 & 0 & d_{\Sigma^2 C}
\end{bmatrix}
\right) .
\]
It remains to show that there is an isomorphism $\alpha_\xi$ from $\Sigma^2(M_\xi)^\vee$ to the totalization of $\xi^{\op}$:
\[
M_{\xi^{\op}}
=\left(
\Sigma^2C\oplus \Sigma B\oplus A,
\begin{bmatrix}
    d_{\Sigma^2C} & 0 & 0\\
    -g^{\op} & d_{\Sigma B} & 0\\
    -h^{\op} & f^{\op} & d_A
\end{bmatrix}
\right).
\]
We set $\alpha_\xi$ to be
\[
\begin{bsmallmatrix}
0 & 0 & \id_{\Sigma^2 C}\\
0 & -\id_{\Sigma B} & 0\\
\id_A & 0 & 0
\end{bsmallmatrix}.
\]
This matrix is invertible, with inverse given by the same matrix, and
a direct calculation gives
\[
d_{M_{\xi^\op}}\alpha_\xi
=
\alpha_\xi d_{\Sigma^2(M_\xi)^\vee}.
\]

Now let
$
\Xi\colon\xi\to\xi'
$
be a homogeneous morphism in $\cpx^3_{\dg}(\A)$.
The explicit description of $\Xi^{\op}$ in
\cref{lem:duality_on_3-term_complexes} shows that
\[
\Tot_{\A^{\op}}(\Xi^{\op})\circ\alpha_{\xi'}
=
\alpha_\xi\circ
\Sigma^2\bigl(\Tot_\A(\Xi)\bigr)^\vee.
\]
Thus the isomorphisms $\alpha_\xi$ are dg natural in $\xi$ and yield the asserted natural isomorphism of dg functors.
\end{proof}

For later use, we introduce the following notation.

\begin{definition}\label{def:shifted_dual_subcategory}
Let $\CM$ be a full subcategory of $\CC_2(\A)$.
We denote by $\CM^\dagger$ the essential image of the composite
\[
\CM^{\op}
\lhook\joinrel\longrightarrow
\CC_2(\A)^{\op}
\xrightarrow{\ \Sigma^2(-)^\vee\ }
\CC_2(\A^{\op}).
\]
\end{definition}

As a direct consequence of \cref{lem:duality_on_defects}, we obtain
the following.

\begin{corollary}\label{cor:shifted_duality_on_defects}
Let $\CS$ be any class of short exact sequences in $\A$, and put $\CS^{\op}=\Set{\xi^{\op}\mid\xi\in\CS}$.
Then the shifted dg $\A$-duality restricts to a duality
\[
\Sigma^2(-)^\vee\colon
(\CM_{\CS})^{\op}
\xrightarrow{\sim}
\CM_{\CS^{\op}}.
\]
In particular, we have $\CM_{\CS}^{\dagger}=\CM_{\CS^{\op}}$.
\end{corollary}

%%%%%%%%%%%%%%%%%%%%%%%%%%%%%%%%%%%%%%%%%%%%%%%%%%%%%%%%%%%%%%%%%
\section{Classifying exact dg structures}
%%%%%%%%%%%%%%%%%%%%%%%%%%%%%%%%%%%%%%%%%%%%%%%%%%%%%%%%%%%%%%%%%

\subsection{An auxiliary correspondence}
Throughout this subsection, we assume that $\A$ is connective and idempotent complete.
As a first step toward relating exact dg structures on $\A$ to suitable full subcategories of $\CC_2(\A)$, we establish an auxiliary correspondence. The argument is a dg analogue of \cite[\S 2.3]{Eno18}.

Let $\CM$ be a full subcategory of $\CC_2(\A)$.
We denote by $\CS_\CM$ the class of all $3$-term complexes
$
\xi=(A\xto{f}B\xto{g}C,h)
$
in $\A$ whose associated totalization $M_\xi$ belongs to $\CM$.
By \cref{lem:defect_vanishing}, every $3$-term complex belonging to
$\CS_\CM$ is a short exact sequence.
Thus, we have the following two maps:
\begin{equation}\label{SM}
\Set{\text{full subcategories of }\CS_{\all}}
\underset{\mathsf{S}}{\overset{\mathsf{M}}{\rightleftarrows}}
\Set{\text{full subcategories of }\CC_2(\A)}
\end{equation}
given by
\[
\mathsf{S}(\CM)=\CS_\CM
\qquad\text{and}\qquad
\mathsf{M}(\CS)=\CM_\CS.
 \]

We briefly review the standard $t$-structure on the derived category of a connective dg category.

\begin{lemma}[{\cite[Lem.~2.9]{Che24b}; see also \cite[Theorem~III.2.3]{BR07}}]
\label{lem:standard_t-str}
Let $\A$ be a connective dg category.
Then the derived category $\CD(\A)$ admits the standard $t$-structure $\bigl(\CD(\A)^{\leq0},\CD(\A)^{\geq0}\bigr)$
given by
\begin{align*}
\CD(\A)^{\leq0}
&=
\Set{M\in\CD(\A)\mid H^i(M)=0\text{ for all }i>0},
\\
\CD(\A)^{\geq0}
&=
\Set{M\in\CD(\A)\mid H^i(M)=0\text{ for all }i<0}.
\end{align*}
Moreover, the zeroth cohomology functor induces an equivalence
\[
H^0\colon
\CH_\A
\deff
\CD(\A)^{\leq0}\cap\CD(\A)^{\geq0}
\xrightarrow{\sim}
\Mod H^0(\A).
\]
\end{lemma}

Throughout the rest of this paper, we often identify $\CH_\A$ with
$\Mod H^0(\A)$ via the above equivalence.

\begin{lemma}[{cf. \cite[Lemma~3.6]{Che24b}}]
\label{lem:defect_in_heart}
Let $\xi=(A\xto{f}B\xto{g}C,h)$ be a $3$-term complex in $\A$ and $M_\xi$ the totalization of $\xi$.
Then the following assertions hold:
\begin{enumerate}
\item 
$\xi$ is left exact if and only if $M_\xi$ belongs to the heart $\CH_\A$.
\item 
$\xi$ is right exact if and only if $\Sigma^2(M_\xi)^\vee$ belongs to the heart $\CH_{\A^\op}$.
\end{enumerate}
\end{lemma}
\begin{proof}
(1)
We have an isomorphism
\[
H^i(M_\xi)(X)\cong \Hom_{\CD(\A)}(X,\Sigma^iM_\xi)
\]
for any $i\in\mathbb{Z}$ and $X\in\A$.
By \cref{lem:defect_vanishing}, the assertion follows.

(2)
Similarly, the shifted dg $\A$-duality yields
\begin{align*}
H^i(\Sigma^2M_\xi^\vee)(X)
&\cong \Hom_{\CD(\A^{\op})}(X,\Sigma^i\Sigma^2M_\xi^\vee) \\
&\cong \Hom_{\CD(\A)}(M_\xi,\Sigma^{i+2}X)
\end{align*}
for any $i\in\mathbb{Z}$ and $X\in\A$.
Thus, $\Sigma^2M_\xi^\vee$ belongs to $\CH_{\A^{\op}}$ precisely
when the last group vanishes for every $i\neq0$.
\end{proof}

For later use, we record the following refinement of \cref{lem:weight_amplitude}.
It allows us to prescribe the last morphism of the resulting $3$-term complex.

\begin{lemma}\label{lem:realization_of_projective_presentation}
Let $M\in\CC_2(\A)$. Every projective presentation
\[
H^0(B)
\xrightarrow{H^0(g)}
H^0(C)
\xrightarrow{p}
M
\longrightarrow0
\]
in $\Mod H^0(\A)$, where $g\colon B\to C$ belongs to $Z^0\A$,
is induced by a short exact $3$-term complex
\[
\xi=(A\xto{f}B\xto{g}C,h).
\]
More precisely, there exists an isomorphism
$
\alpha\colon M_\xi\xrightarrow{\sim}M
$
in $\tr(\A)$ such that
$
H^0(\alpha)\circ p_\xi=p,
$
where $p_\xi\colon H^0(C)\to M_\xi$ is the canonical epimorphism
associated with $\xi$.
\end{lemma}
\begin{proof}
We follow the proof of \cref{lem:weight_amplitude}, keeping track of the prescribed projective presentation.
Regard $p$ as a morphism $p\colon C\to M$ in $\tr(\A)$ and complete
it to a triangle
\[
V\xrightarrow{q}C\xrightarrow{p}M\longrightarrow\Sigma V.
\]
Since $pg=0$, the morphism $g$ lifts to a morphism
$b\colon B\to V$ such that $qb=g$.

For every $X\in\A$, the vanishing $
\Hom_{\tr(\A)}(X,\Sigma^{-1}M)=0
$ shows that the induced morphism
\[
\Hom_{\tr(\A)}(X,V)\overset{q\circ -}{\lra}
\Hom_{\tr(\A)}(X,C)
\]
is injective.
Moreover, the exactness of the prescribed projective
presentation shows that
\[
\Hom_{\tr(\A)}(X,B)\overset{b\circ -}{\lra}
\Hom_{\tr(\A)}(X,V)
\]
is surjective.
Complete $b$ to a triangle
\[
A\xto{f}B\xto{b}V\longrightarrow\Sigma A.
\]
As in the proof of \cref{lem:weight_amplitude}, we have
$V\in\tr(\A)^{\leq0}\cap\tr(\A)^{\geq-1}$.
The above surjectivity then implies
\[
A\in\tr(\A)^{\leq0}\cap\tr(\A)^{\geq0}.
\]
Since $\A$ is idempotent complete, we may regard $A$ as an object of
$\A$.

Identifying $V$ with $\cone(f)$ and choosing representatives, we may
write
\[
q=\begin{bmatrix}-h&g\end{bmatrix}
\colon\cone(f)\longrightarrow C.
\]
This gives a $3$-term complex
$
\xi=(A\xto{f}B\xto{g}C,h)
$
whose totalization is isomorphic to $M$, compatibly with $p$.
Finally, \cref{lem:defect_vanishing} shows that $\xi$ is short exact.
\end{proof}

\begin{definition}\label{def:total_saturated}
A full subcategory $\CS$ of $\CS_{\all}$ is called
\emph{totalization-saturated} if, for any
$\xi,\eta\in\CS_{\all}$ such that
$
M_\xi\cong M_\eta
\quad\text{in }\tr(\A),
$
we have
\[
\xi\in\CS
\quad\Longleftrightarrow\quad
\eta\in\CS.
\]
\end{definition}

\begin{lemma}
Regarded as a full subcategory of $\cpx^3(\A)$, the category
$\CS_{\all}$ is additive and closed under direct summands in
$\cpx^3(\A)$.
\end{lemma}
\begin{proof}
The totalization functor is additive, and hence
$
M_{\xi\oplus\eta}\cong M_\xi\oplus M_\eta
$
for all $\xi,\eta\in\cpx^3(\A)$.

Since $M_0=0$ belongs to $\CC_2(\A)$, the zero $3$-term complex
belongs to $\CS_{\all}$. If $\xi,\eta\in\CS_{\all}$, then
$M_\xi,M_\eta\in\CC_2(\A)$. Since $\CC_2(\A)$ is additive, we have
$
M_\xi\oplus M_\eta
\in\CC_2(\A).
$
Thus, $\xi\oplus\eta\in\CS_{\all}$.

Finally, suppose that $\xi\oplus\eta\in\CS_{\all}$. Then
$
M_{\xi\oplus\eta}
\cong
M_\xi\oplus M_\eta
\in\CC_2(\A)$.
Since $\CC_2(\A)$ is closed under direct summands by definition, both
$M_\xi$ and $M_\eta$ belong to $\CC_2(\A)$. Therefore,
$\xi,\eta\in\CS_{\all}$.
\end{proof}

\begin{corollary}\label{cor:correspondence_defect_subcategories}
The assignments
\[
\CM\longmapsto\CS_\CM,
\qquad
\CS\longmapsto\CM_\CS
\]
induce mutually inverse bijections between the following classes:
\begin{enumerate}
\item
Full subcategories of $\CS_{\all}$ which are
totalization-saturated, additive, and closed under direct summands
in $\cpx^3(\A)$;
\item
Full additive subcategories of $\CC_2(\A)$ which are closed under direct summands in $\tr(\A)$.
\end{enumerate}
\end{corollary}
\begin{proof}
Let $\CM\sse\CC_2(\A)$ belong to the class \textup{(2)}.
Since
$
M_{\xi\oplus\eta}\cong M_\xi\oplus M_\eta,
$
the additivity and closure under direct summands of $\CM$ imply that $\CS_\CM$ is additive and closed under direct summands in $\cpx^3(\A)$.
Moreover, $\CS_\CM$ is totalization-saturated by construction.
Thus $\CS_\CM$ belongs to the class \textup{(1)}.

Now let $\CS$ belong to the class \textup{(1)}.
It is clear that $\CM_\CS$ is additive.
Let $M\in\CM_\CS$ and suppose that
$
M\cong N\oplus N'
$
in $\tr(\A)$. Since $\CC_2(\A)$ is closed under direct summands,
we have $N,N'\in\CC_2(\A)$. By \cref{cor:represents_obj_of_C2}, there exist
$\eta,\eta'\in\CS_{\all}$ such that
\[
M_\eta\cong N
\qquad\text{and}\qquad
M_{\eta'}\cong N'.
\]
Choose $\xi\in\CS$ such that $M\cong M_\xi$. Then
$
M_\xi\cong M_{\eta\oplus\eta'}
$.
Since $\CS$ is totalization-saturated, we have
$\eta\oplus\eta'\in\CS$.
Since $\CS$ is closed under direct summands, it follows that
$\eta,\eta'\in\CS$, and hence $N,N'\in\CM_\CS$.
Thus $\CM_\CS$ belongs to the class \textup{(2)}.

For $\CM$ in \textup{(2)}, we have
$
\CM_{\CS_\CM}=\CM
$
by \cref{cor:represents_obj_of_C2} and the definitions.
On the other hand, for $\CS$ in \textup{(1)}, the fact that $\CS$
is totalization-saturated yields $\CS_{\CM_\CS}=\CS$.
Therefore the two assignments are mutually inverse.
\end{proof}

We show that every morphism in $\CC_2(\A)$ is represented by a morphism of short exact sequences.

\begin{corollary}\label{cor:represents_morph_of_C2}
Let $\phi\colon M\to M'$ be a morphism in $\CC_2(\A)$.
Then there exists a morphism $\Xi\colon \xi\to\xi'$ in $\cpx^3(\A)$ such that $\Tot_\A(\Xi)\cong \phi$.
\end{corollary}
\begin{proof}
By \cref{cor:represents_obj_of_C2}, there exist short exact sequences
\[
\xi=(A\xto{f}B\xto{g}C,h)
\qquad
\text{and}
\qquad
\xi'=(A'\xto{f'}B'\xto{g'}C',h')
\]
such that $M\cong M_\xi$ and $M'\cong M_{\xi'}$ in $\CC_2(\A)$.
In turn we may replace $M,M'$ with the isomorphic objects $M_{\xi},M_{\xi'}$ respectively.
We have seen in \cref{lem:defect_in_heart} that $\CC_2(\A)$ is contained in the heart $\CH_\A=\CD(\A)^{\leq 0}\cap \CD(\A)^{\geq 0}$.
Via the equivalence
\[
H^0\colon\CH_\A\xrightarrow{\sim}\Mod H^0(\A)
\]
of \cref{lem:standard_t-str}, we regard $M_\xi$ and $M_{\xi'}$ as
$H^0(\A)$-modules.
Thus we deduce from the defining diagram \eqref{diag:defect} of $M_\xi$ the triangle in $\tr(\A)$,
\[
\begin{tikzcd}
U_{\xi}\arrow{r}{\begin{bsmallmatrix}
    -h, \amph g
\end{bsmallmatrix}}
&C\arrow{r}
&M_{\xi}\arrow{r}
&\Sigma U_{\xi}
\end{tikzcd}
\]
where $\cone(f)=U_{\xi}$.
By connectivity, we have projective presentations of $M_{\xi}$ and $M_{\xi'}$, see \cite[Remark~3.11]{Che24b} for the details.
We have thus obtained the following commutative diagram in $\Mod H^0(\A)$:
\[
\begin{tikzcd}
H^0(B)\arrow{r}{H^0(g)}\arrow{d}{}\arrow[mysymbol]{rd}[description]{\circlearrowleft}
&H^0(C)\arrow{r}\arrow{d}{}
&M_\xi\arrow{r}\arrow{d}{\phi}
&0
\\
H^0(B')\arrow{r}[swap]{H^0(g')}
&H^0(C')\arrow{r}
&M_{\xi'}\arrow{r}
&0
\end{tikzcd}
\]
By choosing closed degree-zero representatives of the unlabeled vertical arrows, we have the following closed degree-zero morphism from $g$ to $g'$ in $\Mor(\A)$:
\[
\begin{tikzcd}
B\arrow{r}{g}\arrow{d}[swap]{r_1} \arrow{rd}[red]{s_2}&C\arrow{d}{r_2}
\\
B'\arrow{r}{g'} &C'
\end{tikzcd}
\]
which corresponds to the $\circlearrowleft$-labeled square.
Since $\xi'$ is left exact,
\cite[Lemma~3.29]{Che24a} shows that this morphism extends to a
morphism
\[
\Xi=(r_0,r_1,r_2,s_1,s_2,t)
\colon
\xi\to\xi'
\]
in $\cpx^3(\A)$.
By construction, the morphism on cokernels induced by $\Tot_\A(\Xi)$ is precisely $\phi$.
Equivalently,
\[
H^0\bigl(\Tot_\A(\Xi)\bigr)=\phi.
\]
Since the heart $\CH_\A$ is fully faithfully embedded in $\CD(\A)$, we conclude that $\Tot_\A(\Xi)=\phi$ in $\tr(\A)$.
\end{proof}

We conclude this subsection with a characterization of stable dg categories in terms of $\CC_2(\A)$ and the shifted dg duality.

\begin{definition}\label{def:stable_dg_category}
\cite[Definition~6.1]{Che24a}
A connective dg category $\A$ is said to be \emph{stable} if the
following conditions hold:
\begin{enumerate}
\item
The dg category $\A$ admits homotopy kernels and homotopy
cokernels.
\item
A $3$-term complex in $\A$ is left exact if and only if it is
right exact.
\end{enumerate}
\end{definition}

Let $\A$ be a stable dg category.
By \cite[Proposition~3.28]{Che24b}, $\A$ is quasi-equivalent to the connective cover of a pretriangulated dg category.
By \cite[Proposition~6.4]{Che24a}, the class $\CS_{\all}$ defines an exact dg structure on $\A$.

\begin{proposition}\label{prop:shifted_duality_stable}
Let $\A$ be an idempotent complete connective additive dg category.
Then the following conditions are equivalent:
\begin{enumerate}
\item
$\A$ is a stable dg category.
\item
The zeroth cohomology functor induces an equivalence
$
H^0\colon
\CC_2(\A)\xrightarrow{\sim}\mod H^0(\A),
$
and the shifted dg $\A$-duality induces a duality
\begin{equation}\label{equiv:shifted_duality_stable}
\Sigma^2(-)^\vee\colon
\bigl(\mod H^0(\A)\bigr)^{\op}
\xrightarrow{\sim}
\mod H^0(\A^{\op}).
\end{equation}
\end{enumerate}
\end{proposition}
\begin{proof}
(1) $\Rightarrow$ (2):
Since $\A$ is stable, $\CS_{\all}$ is an exact dg structure and the
induced extriangulated structure on $H^0(\A)$ is triangulated, see \cite[Theorem~6.5]{Che24a}.
Moreover, the zeroth cohomology functor induces an equivalence
\[
H^0\colon
\CC_2(\A)=\CM_{\CS_{\all}}
\xto{\sim}
\defect\BE_{\CS_{\all}}.
\]
Every finitely presented $H^0(\A)$-module admits a projective
presentation
\[
H^0(\A)(-,B)\lra H^0(\A)(-,C)
\lra M\lra 0.
\]
Completing the corresponding morphism $B\to C$ to a triangle shows
that $M$ is its contravariant defect. Hence
\[
\defect\BE_{\CS_{\all}}=\mod H^0(\A).
\]
Since the triangulated category $H^0(\A)$ has weak kernels, we also
have $\mod H^0(\A)=\coh H^0(\A)$.
The equivalence \eqref{equiv:shifted_duality_stable} follows from
\cref{cor:shifted_duality_on_defects}.

(2) $\Rightarrow$ (1):
Let $\xi=(A\xto{f}B\xto{g}C,h)$ be a left exact sequence
in $\A$.
By \cref{lem:defect_in_heart}(1), its totalization $M_\xi$ belongs
to $\Mod H^0(\A)$. Moreover, the canonical presentation
\[
H^0(B)\xrightarrow{H^0(g)}H^0(C)\longrightarrow M_\xi
\longrightarrow0
\]
shows that $M_\xi$ belongs to $\mod H^0(\A)$.
By \eqref{equiv:shifted_duality_stable},
$\Sigma^2(M_\xi)^\vee$ belongs to $\mod H^0(\A^{\op})$.
It follows from \cref{lem:defect_in_heart}(2) that $\xi$ is right
exact.
Applying the same argument to $\A^{\op}$ proves the converse.
Thus left and right exactness coincide.

It remains to prove the existence of kernels and cokernels.
Let $g\colon B\to C$ be a closed degree-zero morphism in $\A$, and
consider the projective presentation
\[
H^0(B)\overset{H^0(g)}{\lra}
H^0(C)\lra M\lra 0.
\]
By the first equivalence in \textup{(2)}, we may regard
$M$ as an object of $\CC_2(\A)$.
Hence \cref{lem:realization_of_projective_presentation} gives a
short exact $3$-term complex
\[
\xi=(A\xto{f}B\xto{g}C,h).
\]
In particular, $g$ admits a kernel.

By \cref{lem:shifted_duality_on_C2} and
\eqref{equiv:shifted_duality_stable}, we also have
\[
\CC_2(\A^{\op})
=
\CC_2(\A)^\dagger
=
\mod H^0(\A^{\op}).
\]
Applying the preceding argument to $\A^{\op}$ shows that every
closed degree-zero morphism in $\A$ admits a cokernel.
Therefore $\A$ is stable.
\end{proof}

\begin{remark}
The characterization in
\cref{prop:shifted_duality_stable} is the case $d=1$ of Tomonaga's Auslander correspondence \cite[Theorem~0.1]{Tom26}, reformulated in terms of the intrinsically defined subcategory $\CC_2(\A)$.
In this formulation, the coherence and weak-global-dimension conditions appearing in Tomonaga's theorem are encoded by the equivalence
\[
\CC_2(\A)\simeq\mod H^0(\A).
\]
\end{remark}

\subsection{Proof of the classification theorem}

We begin by introducing the notion needed to state our classification theorem.

\begin{definition}\label{def:bi-Serre}
Let $\CM$ be a full subcategory of $\CC_2(\A)$.
We call $\CM$ \emph{bi-Serre} if $\CM$ is a Serre subcategory of $\coh H^0(\A)$ and its shifted dual $\CM^\dagger$, defined in
\cref{def:shifted_dual_subcategory},
is a Serre subcategory of $\coh H^0(\A^{\op})$.
\end{definition}

\begin{theorem}\label{thm:correspondence}
Let $\A$ be a connective additive dg category and assume that $\A$
is idempotent complete.
Then the assignments $\mathsf{S}$ and $\mathsf{M}$ in \eqref{SM}
induce mutually inverse bijections between the following classes:
\begin{enumerate}[label=\textup{(\alph*)}]
\item\label{item:exact_dg_structures}
exact dg structures on $\A$;
\item\label{item:bi_serre_subcategories}
bi-Serre subcategories of $\CC_2(\A)$.
\end{enumerate}
\end{theorem}

We first show that the class \ref{item:exact_dg_structures} is precisely the subclass of the class \textup{(1)} in \cref{cor:correspondence_defect_subcategories} consisting of those $\CS$ which satisfy
\ref{EX0}--\ref{EX2} and \ref{EX2op}.
Our aim is then to show that the bijection established in
\cref{cor:correspondence_defect_subcategories} restricts to a bijection between this subclass and the class \ref{item:bi_serre_subcategories}:
\begin{equation}\label{restricted_SM}
\Set{\text{exact dg structures on }\A}
\underset{\mathsf{S}}{\overset{\mathsf{M}}{\rightleftarrows}}
\Set{\text{bi-Serre subcategories of }\CC_2(\A)}.
\end{equation}

The proof is rather lengthy, so we divide it into several steps.
We begin by establishing some closure properties of exact dg
structures.

\begin{lemma}\label{lem:closed_under_direct_summands}
Let $\CS$ be an exact dg structure on a dg category $\A$.
Then $\CS$ is additive and closed under direct summands in $\cpx^3(\A)$.
\end{lemma}
\begin{proof}
By \ref{EX0}, the identity morphism of the zero object is a deflation,
and hence the zero $3$-term complex belongs to $\CS$.
Let
$$
\xi=(A\xto{f}B\xto{g}C,h)
\qquad\text{and}\qquad
\xi'=(A'\xto{f'}B'\xto{g'}C',h')
$$
be short exact sequences in $\CS_{\all}$.
Recall that $\xi\oplus\xi'$ also belongs to $\CS_{\all}$.

Suppose first that $\xi,\xi'\in\CS$.
By \ref{EX2}, the morphisms
$$
g\oplus\id_{B'}
\colon
B\oplus B'\longrightarrow C\oplus B'
$$
and
$$
\id_C\oplus g'
\colon
C\oplus B'\longrightarrow C\oplus C'
$$
are deflations in $\CS$.
By \ref{EX1}, their composite
$$
g\oplus g'
=
\begin{psmallmatrix}
g&0\\
0&g'
\end{psmallmatrix}
=
\begin{psmallmatrix}
\id_C&0\\
0&g'
\end{psmallmatrix}
\circ
\begin{psmallmatrix}
g&0\\
0&\id_{B'}
\end{psmallmatrix}
$$
is again a deflation in $\CS$.
It follows from the uniqueness of kernels that $\xi\oplus\xi'$ belongs to $\CS$; see also \cite[p.~38, after Lemma~3.29]{Che24a}.
Thus $\CS$ is additive.

Next, suppose that $\xi\oplus\xi'\in\CS$.
By the proof of \cite[Lemma~4.24]{Che24a}, we have an isomorphism
$$
(p_A)_*\iota_C^*(\xi\oplus\xi')
\cong
\xi
$$
in $\cpx^3(\A)$, where
$$
p_A\colon A\oplus A'\longrightarrow A
\qquad\text{and}\qquad
\iota_C\colon C\longrightarrow C\oplus C'
$$
are the canonical projection and inclusion, respectively.
By \ref{EX2} and \ref{EX2op}, it follows that $\xi\in\CS$.
Therefore $\CS$ is closed under direct summands in $\cpx^3(\A)$.
\end{proof}

\begin{lemma}\label{lem:totalization_saturated}
Let $\CS$ be an exact dg structure on a dg category $\A$.
Then $\CS$ is totalization-saturated.
\end{lemma}
\begin{proof}
Let
\[
\xi=(A\xto{f}B\xto{g}C,h)\in\CS_{\all}
\]
and
\[
\xi'=(A'\xto{f'}B'\xto{g'}C',h')\in\CS,
\]
and suppose that there is an isomorphism
$
\phi\colon M_\xi\xrightarrow{\sim}M_{\xi'}
$
in $\CC_2(\A)$.
By \cref{cor:represents_morph_of_C2}, the isomorphism $\phi$ lifts to a morphism
\[
\Xi=(r_0,r_1,r_2,s_1,s_2,t)
\colon
\xi\longrightarrow\xi'
\]
in $\cpx^3(\A)$ such that
$
\Tot_\A(\Xi)=\phi
$.
In particular, its restrictions to the first and third terms are
\[
r_0\colon A\longrightarrow A'
\qquad\text{and}\qquad
r_2\colon C\longrightarrow C',
\]
respectively.
Let
\[
\eta=(A'\longrightarrow D\longrightarrow C)
\]
be the pullback of $\xi'$ along $r_2$.
Since $\xi'\in\CS$, axiom \ref{EX2} gives
$\eta=r_2^*\xi'\in\CS$.
Thus, we obtain the last two rows of the following diagram:
\[
\begin{tikzcd}
A\arrow{r}{f}\arrow{d}[swap]{r_0}
&B\arrow{r}{g}\arrow[dotted]{d}
&C\arrow[equal]{d}
\\
A'\arrow{r}\arrow[equal]{d}
&D\arrow{r}\arrow{d}
&C\arrow{d}{r_2}
\\
A'\arrow{r}{f'}
&B'\arrow{r}{g'}
&C'.
\end{tikzcd}
\]
By the universal property of the pullback, or more precisely
by \cite[Lemma~3.36]{Che24a}, the morphism $\Xi$ factors as
\[
\xi\xto{\beta}\eta\xto{\gamma}\xi',
\]
where the restrictions of $\beta$ to the first and third terms are
$r_0$ and $\id_C$, respectively. This is the construction underlying
\cite[Lemma~4.10]{Che24a}.
Since both $\xi$ and $\eta$ are short exact and the
restriction of $\beta$ to their third terms is the identity, the dual
of the argument of \cite[Lemma~4.12]{Che24a} shows that the induced square
\[
\begin{tikzcd}
A\arrow{r}{f}\arrow{d}[swap]{r_0}
&B\arrow{d}
\\
A'\arrow{r}
&D
\end{tikzcd}
\]
is a pushout square.
Consequently, $\eta$ may also be
identified with the pushout of $\xi$ along $r_0$:
\[
\eta\cong(r_0)_*\xi.
\]

We next compare the totalizations of $\xi$, $\eta$, and $\xi'$.
Put
\[
a=\Tot_\A(\beta)\colon M_\xi\longrightarrow M_\eta
\qquad\text{and}\qquad
b=\Tot_\A(\gamma)\colon M_\eta\longrightarrow M_{\xi'}.
\]
Since $\Xi=\gamma\circ\beta$, we have
$b\circ a=\Tot_\A(\Xi)=\phi$.
In particular,
\[
\phi^{-1}\circ b\colon M_\eta\longrightarrow M_\xi
\]
is a retraction of $a$, so that $a$ is a split monomorphism.

On the other hand, the restriction of $\beta$ to the third terms is
the identity of $C$. Consider the resulting commutative diagram in
$\Mod H^0(\A)$:
\[
\begin{tikzcd}
H^0(B)\arrow{r}\arrow{d}
&H^0(C)\arrow{r}\arrow[equal]{d}
&M_\xi\arrow{r}\arrow{d}{a}
&0
\\
H^0(D)\arrow{r}
&H^0(C)\arrow{r}
&M_\eta\arrow{r}
&0.
\end{tikzcd}
\]
It follows that $a$ is an epimorphism. Equivalently, this also follows from the argument of \cite[Lemma~3.9]{Che24b}.
Since $a$ is both a split monomorphism and an epimorphism in the heart $\CH_\A$, it is an isomorphism. Therefore,
\[
M_\xi\cong M_\eta.
\]
It also follows from $b\circ a=\phi$ that $b$ is an isomorphism, and hence $M_\eta\cong M_{\xi'}$.
We have therefore reduced the remaining argument to the following
special case:
\[
\xi=(A\xto{f}B\xto{g}C,h)\in\CS_{\all},
\qquad
\eta=(A'\longrightarrow D\longrightarrow C)\in\CS,
\]
together with a morphism
$
\beta\colon\xi\to\eta
$
whose restriction to the third terms is $\id_C$ and whose totalization
\[
\Tot_\A(\beta)\colon M_\xi\xto{\sim}M_\eta
\]
is an isomorphism.
Thus, after replacing $\xi'$ by $\eta$, we may assume from the outset that
\[
C=C',
\qquad
r_2=\id_C,
\qquad
\xi'\in\CS,
\qquad
\Tot_\A(\Xi)\text{ is an isomorphism}.
\]
It remains to prove $\xi\in\CS$ under these assumptions.
\[
\begin{tikzcd}
A\arrow{r}{f}\arrow{d}[swap]{r_0}
&B\arrow{r}{g}\arrow{d}{r_1}
&C\arrow[equal]{d}
\\
A'\arrow{r}{f'}
&B'\arrow{r}{g'}
&C
\end{tikzcd}
\]
Under the above assumptions, the morphism
$
\Tot_\A(\Xi)\colon M_\xi\to M_{\xi'}
$
is induced by the identity morphism of $H^0(C)$. Indeed, we have a
commutative diagram in $\Mod H^0(\A)$
\[
\begin{tikzcd}
H^0(B)\arrow{r}{H^0(g)}\arrow{d}{H^0(r_1)}
&H^0(C)\arrow{r}\arrow[equal]{d}
&M_\xi\arrow{r}\arrow{d}{\Tot_\A(\Xi)}
&0
\\
H^0(B')\arrow{r}{H^0(g')}
&H^0(C)\arrow{r}
&M_{\xi'}\arrow{r}
&0.
\end{tikzcd}
\]
In particular, $\Im H^0(g)\subseteq\Im H^0(g')$.
Since $\Tot_\A(\Xi)$ is an isomorphism, the induced quotient map is an isomorphism,
and hence
\[
\Im H^0(g)=\Im H^0(g')
\]
as submodules of $H^0(C)$.
Evaluating this equality at $B'$, we obtain
\[
[g']
\in
\Im\bigl(
H^0(\A)(B',B)
\xrightarrow{\,H^0(g)\circ-\,}
H^0(\A)(B',C)
\bigr).
\]
Thus, there exists a morphism
\[
q\colon B'\to B
\]
in $H^0(\A)$ such that
$
[gq]=[g']
$
in $H^0(\A)$.
After choosing a closed representative of $q$, the morphisms $gq$ and $g'$ are homotopic.
Since $\xi'\in\CS$, the morphism $g'$ is a deflation.
Hence $gq$ is also a deflation; see \cite[Remark~4.4]{Che24a}.
Moreover, since $\xi\in\CS_{\all}$, the morphism $g$ admits a kernel.
Therefore, \cite[Proposition~4.11(b)]{Che24a} implies that $g$ is a deflation.

Consequently, $g$ admits a kernel belonging to $\CS$.
On the other hand, $\xi$ is itself a kernel of $g$.
By the uniqueness of kernels, $\xi$ is isomorphic in $\cpx^3(\A)$ to a conflation in $\CS$.
Since $\CS$ is closed under isomorphisms, we conclude that
$$
\xi\in\CS.
$$
This proves that $\CS$ is totalization-saturated.
\end{proof}

\begin{lemma}\label{lem:exact_implies_bi_serre}
Let $\CS$ be an exact dg structure on $\A$.
Then $\mathsf{M}(\CS)=\CM_\CS$ is a bi-Serre subcategory of $\CC_2(\A)$.
\end{lemma}
\begin{proof}
By \cref{thm:exact_dg_to_extri}, the exact dg structure $\CS$ determines an extriangulated structure $(\BE_\CS,\fs_\CS)$ on $H^0(\A)$.
Under the equivalence
\[
H^0\colon\CH_\A\xrightarrow{\sim}\Mod H^0(\A)
\]
of \cref{lem:standard_t-str},
the subcategory $\CM_\CS$ identifies with the defect subcategory $\defect\BE_\CS$ of \cref{def:defect_subcategory}.
By \cref{prop:defect_subcat_is_Serre}, $\defect\BE_\CS$ is a Serre subcategory of $\coh H^0(\A)$.
By the symmetry of the exact dg structure, we have an exact dg category $(\A^{\op},\CS^{\op})$ and the associated extriangulated category 
$\bigl(H^0(\A^{\op}),\BE_{\CS^{\op}},\fs_{\CS^{\op}}\bigr)$.
Again, by \cref{prop:defect_subcat_is_Serre}, we see that
$\CM_{\CS^{\op}}$ is a Serre subcategory of
$\coh H^0(\A^{\op})$.
By \cref{cor:shifted_duality_on_defects}, we have $\CM_\CS^\dagger=\CM_{\CS^{\op}}$.
Hence $\CM_\CS$ is a bi-Serre subcategory of $\CC_2(\A)$.
\end{proof}

By \cref{lem:exact_implies_bi_serre}, the assignment $\mathsf{M}$ restricts to a map
\[
\mathsf{M}\colon
\Set{\text{exact dg structures on }\A}
\longrightarrow
\Set{\text{bi-Serre subcategories of }\CC_2(\A)}.
\]
On the other hand,
by \cref{lem:closed_under_direct_summands,lem:totalization_saturated},
every exact dg structure belongs to the class \textup{(1)} of \cref{cor:correspondence_defect_subcategories}.
Therefore, that corollary yields
\[
\mathsf{S}\circ\mathsf{M}=\id
\]
on the class of exact dg structures on $\A$.

\begin{lemma}\label{lem:bi_serre_implies_exact}
Let $\CM$ be a bi-Serre subcategory of $\CC_2(\A)$.
Then $\mathsf{S}(\CM)=\CS_\CM$ is an exact dg structure on $\A$.
\end{lemma}
\begin{proof}
We first make a basic observation which will be used repeatedly.
Let $M\in\CM$ and let
\[
v\colon H^0(X)\longrightarrow M
\]
be a morphism in $\Mod H^0(\A)$.
Its image is a finitely generated submodule of the coherent module $M$, and hence is coherent.
Since $\CM$ is a Serre subcategory of $\coh H^0(\A)$, both $\Im v$ and $\Cok v$ belong to $\CM$.

By \cref{lem:defect_vanishing}, the class $\CS_\CM$ consists of short exact sequences.
Moreover, it is closed under isomorphisms since totalization preserves isomorphisms and $\CM$ is closed under isomorphisms.
It remains to verify the axioms \ref{EX0}--\ref{EX2} and \ref{EX2op}.

\ref{EX0}: 
Consider a split short exact sequence
\[
\xi_A=
(0\longrightarrow A\xrightarrow{\id_A}A,0).
\]
Its totalization $M_{\xi_A}$ is contractible and hence isomorphic to
zero in $\tr(\A)$.
Hence $M_{\xi_A}\in\CM$ and $\xi_A\in\CS_\CM$.
Thus $\id_A$ is a deflation in $\CS_\CM$.

\ref{EX2}:
Let $\xi=(A\xto{f}B\xto{g}C,h)$ be a short exact sequence with $M_{\xi}\in\CM$ and $u\colon C'\to C$ a morphism in $Z^0\A$.
Under the cohomological functor $H^0\colon \CD(\A)\to \Mod H^0(\A)$, we have a projective presentation
\begin{equation*}
H^0(B)\xto{H^0(g)} H^0(C)\xto{p_1} M_\xi\to 0.
\end{equation*}
We consider the factor module $N$ of $M_\xi$ defined by the following exact sequence:
\[
H^0(C')\oplus H^0(B)\xto{\begin{bsmallmatrix}
    H^0(-u), & H^0(g)
\end{bsmallmatrix}} H^0(C)\xto{p_2}N\to 0.
\]
Equivalently, $N$ is the factor module of $M_\xi$ by the image of the
morphism
\[
p_1\circ H^0(u)\colon H^0(C')\longrightarrow M_\xi.
\]
Hence $N\in\CM$ by the preceding observation.
Applying the construction in the proof of
\cref{lem:realization_of_projective_presentation} to the prescribed
projective presentation, there exists a bicartesian square
\[
\begin{tikzcd}
B'\arrow{r}{g'}\arrow{d}[swap]{u'}\arrow[red]{rd}{s} & C'\arrow{d}{u}\\
B\arrow{r}{g} & C 
\end{tikzcd}
\]
such that the folding is a short exact sequence $\eta$ and is a realization of $N$ in the sense that $M_{\eta}\cong N$.
Since the square is homotopy cartesian, the dual of
\cite[Corollary~3.35]{Che24a} shows that it extends to a morphism
of $3$-term complexes
\[
\begin{tikzcd}
A\arrow{r}{f'}\arrow[equal]{d}
&B'\arrow{r}{g'}\arrow{d}[swap]{u'}\arrow[red]{rd}{s}
&C'\arrow{d}{u}\\
A\arrow{r}{f}
&B\arrow{r}{g}
&C,
\end{tikzcd}
\]
whose restriction to the first terms is $\id_A$.
Hence the top row gives a short exact sequence $\xi'$ which is a
homotopy pullback of $\xi$ along $u$. By the dual of
\cite[Lemma~3.9]{Che24b}, there is an exact sequence
\[
0\longrightarrow M_{\xi'}
\longrightarrow M_\xi
\longrightarrow M_\eta
\longrightarrow0
\]
in $\coh H^0(\A)$.
Since $M_\xi\in\CM$ and $\CM$ is a Serre subcategory, we have
$M_{\xi'}\in\CM$. Therefore, $\xi'\in\CS_\CM$.

\ref{EX2op}:
Apply the argument for \ref{EX2} to the dg category $\A^{\op}$ and
the Serre subcategory $\CM^\dagger$ of $\coh H^0(\A^{\op})$.
By \cref{lem:duality_on_defects} and the definitions, we have
\[
(\CS_\CM)^{\op}=\CS_{\CM^\dagger}.
\]
Thus, axiom \ref{EX2} for $\CS_{\CM^\dagger}$ in $\A^{\op}$ is
precisely axiom \ref{EX2op} for $\CS_\CM$ in $\A$.

\ref{EX1}:
Let the following be short exact sequences
\[
\xi=(A\xto{f}B\xto{g}C,h)
\qquad
\text{and}
\qquad
\xi'=(A'\xto{f'}C\xto{g'}C',h')
\]
with $M_\xi, M_{\xi'}\in\CM$.
Similarly to the proof of \ref{EX2}, we consider the factor module $N$ of $M_\xi$ defined by the exactness of the following sequence
\[
H^0(A')\oplus H^0(B)\xto{\begin{bsmallmatrix}
    H^0(-f'), & H^0(g)
\end{bsmallmatrix}} H^0(C)\xto{p}N\to 0
\]
in $\Mod H^0(\A)$.
Since $N\in\CM$, we also have a short exact sequence
\[
\eta=\bigl(
A''\xto{\begin{bsmallmatrix}
    r_0 \\ f''
\end{bsmallmatrix}}A'\oplus B\xto{\begin{bsmallmatrix}
    -f', & g
\end{bsmallmatrix}}C,s
\bigr)
\]
with an isomorphism $N\cong M_\eta$.
Thanks to \cite[Lemma~3.37]{Che24a}, we have a morphism to $\xi'$:
\[
\begin{tikzcd}[row sep=0.8cm, column sep=1.2cm]
A''
 \arrow{r}{f''}
 \arrow[bend left]{rr}{h''}
 \arrow{d}[swap]{r_0}
 \arrow[draw=red]{rd}[swap]{\textcolor{red}{s}}
 \arrow[draw=blue]{rrd}{\textcolor{blue}{t}}
&B
 \arrow{r}{g'g}
 \arrow{d}[swap]{g}
 \arrow[draw=red]{rd}{\textcolor{red}{0}}
&C'\arrow[equal]{d}{}
\\
A'
 \arrow{r}[swap]{f'}
 \arrow[bend right]{rr}[swap]{h'}
&C\arrow{r}[swap]{g'}
&C'.
\end{tikzcd}
\]
and the top row $\xi''$ is a short exact sequence.
By \cite[Lemma~3.9]{Che24b}, this induces a short exact sequence
\[
0\to M_\eta\to M_{\xi''}\to M_{\xi'}\to 0
\]
in $\Mod H^0(\A)$.
We conclude $M_{\xi''}\in\CM$ and $\xi''\in\CS_\CM$.
In particular, the composite $g'g$ is a deflation.
\end{proof}

\begin{proof}[Proof of \cref{thm:correspondence}]
By \cref{lem:bi_serre_implies_exact}, the assignment $\mathsf{S}$ restricts to a map
\[
\mathsf{S}\colon
\Set{\text{bi-Serre subcategories of }\CC_2(\A)}
\longrightarrow
\Set{\text{exact dg structures on }\A}.
\]
Since every bi-Serre subcategory is additive and closed under direct summands, \cref{cor:correspondence_defect_subcategories} yields $\mathsf{M}\circ\mathsf{S}=\id$.
Together with the equality $\mathsf{S}\circ\mathsf{M}=\id$ established above, this completes the lengthy proof.
\end{proof}

%%%%%%%%%%%%%%%%%%%%%%%%%%%%%%%%%%%%%%%%%%%%%%%%%%%%%%%%%%%%%%%%%
\subsection{The lattice of exact dg structures}
\label{subsec:lattice_exact_dg}
%%%%%%%%%%%%%%%%%%%%%%%%%%%%%%%%%%%%%%%%%%%%%%%%%%%%%%%%%%%%%%%%%

We now apply \cref{thm:correspondence} to study the poset of exact dg
structures on $\A$.

\begin{definition}\label{def:poset_on_A}
We denote by
\[
\Ex_{\dg}(\A)
\]
the poset of exact dg structures on $\A$, ordered by inclusion.
Thus, for $\CS,\CS'\in\Ex_{\dg}(\A)$, we write
\[
\CS\leq\CS'
\quad\Longleftrightarrow\quad
\CS\subseteq\CS'.
\]
\end{definition}

For the remainder of this subsection, we assume that $\A$ satisfies the assumptions of \cref{thm:correspondence}.

By \cref{thm:correspondence}, the assignments $\mathsf M$ and $\mathsf S$ induce an isomorphism of posets between $\Ex_{\dg}(\A)$ and the poset of bi-Serre subcategories of $\CC_2(\A)$.

\begin{proposition}\label{prop:lattice_of_exact_dg_structures}
The poset $\Ex_{\dg}(\A)$ is a complete lattice. In particular, it
admits a greatest element $\CS_{\max}$.
\end{proposition}
\begin{proof}
In view of \cref{thm:correspondence}, we only show that the poset of bi-Serre subcategories of $\CC_2(\A)$ has a greatest element.
Put
$$
\CU
=
\bigcup_{\CM\text{: bi-Serre}}\CM
$$
and let $\Filt(\CU)$ be the full subcategory of $\coh H^0(\A)$
consisting of the objects $N$ admitting a finite filtration
$$
0=N_0\sse N_1\sse\cdots\sse N_k=N
$$
such that
$
N_{i+1}/N_i\in\CU
$ for each $0\leq i\leq k-1$.
It is basic that $\Filt(\CU)$ is a Serre subcategory of $\coh H^0(\A)$.
The containment
$
\Filt(\CU)\sse\CC_2(\A)
$
follows immediately from the defining vanishing conditions in \cref{def:C2}.
Since the shifted dg duality is exact on the corresponding hearts, we
have

$$
\Filt(\CU)^\dagger
=
\Filt(\CU^\dagger),
\qquad
\CU^\dagger
=
\bigcup_{\CM\text{: bi-Serre}}\CM^\dagger.
$$
The same argument shows that $\Filt(\CU^\dagger)$ is a Serre
subcategory of $\coh H^0(\A^{\op})$. Therefore,
$$
\CM_{\max}\deff\Filt(\CU)
$$
is a bi-Serre subcategory of $\CC_2(\A)$.
Also, we have a greatest exact dg structure $\CS_{\max}=\CS_{\CM_{\max}}$ of $\Ex_{\dg}(\A)$.
\end{proof}

The existence of the greatest exact dg structure was previously proved in \cite[Theorem~7.10]{Che23}; see also \cite[Theorem~4.18]{Che25survey}.
Thus, under the assumptions of the present section, \cref{prop:lattice_of_exact_dg_structures} provides
an alternative proof of this existence result.
Our argument is independent of Chen's construction of the greatest exact dg structure and follows Enomoto's classification strategy directly at the dg
level.

We should mention that the following Rump--Chen classification theorem follows immediately from \cite[Theorem~4.37]{Che24a} and \cite[Theorem~7.10]{Che23},
and takes a form slightly different from \cref{thm:correspondence}.

Let $\CS_{\max}$ denote the greatest exact dg structure on $\A$.
By \cref{thm:exact_dg_to_extri}, it induces an extriangulated structure
$
(\BE_{\CS_{\max}},\fs_{\CS_{\max}})
$
on $H^0(\A)$.
Moreover, by
\cref{prop:defect_subcat_is_Serre}, its defect category $\defect\BE_{\CS_{\max}}$ is a Serre subcategory of $\coh H^0(\A)$.

\begin{theorem}\label{thm:correspondence_Chen}
Let $\A$ be a connective additive dg category.
Then the following posets are isomorphic:
\begin{enumerate}[label=\textup{(\alph*)}]
\item
the poset of exact dg structures on $\A$;
\item[\textup{(b')}]
the poset of Serre subcategories of
$\defect\BE_{\CS_{\max}}$.
\end{enumerate}
\end{theorem}
\begin{proof}
Since $\CS_{\max}$ is the greatest exact dg structure on $\A$, every exact dg structure on $\A$ is an exact dg substructure of $(\A,\CS_{\max})$.
The assertion therefore follows immediately from \cref{prop:bijection_between_substructures}.
\end{proof}

Comparing \cref{thm:correspondence} with
\cref{thm:correspondence_Chen}, we obtain the following
characterization of the Serre subcategories of
$\defect\BE_{\CS_{\max}}$.

\begin{corollary}\label{cor:bi-Serre_vs_Serre}
Let $\A$ satisfy the assumptions of \cref{thm:correspondence}, and let
$\CM$ be a full subcategory of $\CC_2(\A)$. Then the following
conditions are equivalent:
\begin{enumerate}
\item
$\CM$ is a bi-Serre subcategory of $\CC_2(\A)$;
\item
$\CM$ is contained in $\defect\BE_{\CS_{\max}}$ and is a Serre
subcategory thereof.
\end{enumerate}
In particular, $\defect\BE_{\CS_{\max}}$ is the greatest bi-Serre
subcategory of $\CC_2(\A)$.
\end{corollary}

As an immediate consequence of \cref{thm:correspondence}, we obtain the following criterion.

\begin{corollary}\label{cor:case_of_all=max}
The following conditions are equivalent:
\begin{enumerate}
\item
$\CS_{\max}=\CS_{\all}$.
\item
$\CC_2(\A)$ is a bi-Serre subcategory of itself.
\end{enumerate}
\end{corollary}
\begin{proof}
By definition, we have $\mathsf{S}\bigl(\CC_2(\A)\bigr)=\CS_{\all}$.
The assertion therefore follows from \cref{thm:correspondence}, since
$\CS_{\max}$ corresponds to the greatest bi-Serre subcategory of $\CC_2(\A)$.
\end{proof}

\begin{remark}
The equality $\CS_{\max}=\CS_{\all}$ holds in several standard situations.
\begin{enumerate}
\item 
If $\A$ is stable, then $(\A,\CS_{\all})$ is an exact dg category
by \cite[Proposition~6.4]{Che24a}, and hence
$\CS_{\max}=\CS_{\all}$.
\item
If $\A$ is an abelian $d$-truncated dg category, then $(\A,\CS_{\all})$ is an exact dg category by \cite[Proposition~3.14]{Moc25},
and hence $\CS_{\max}=\CS_{\all}$.
\end{enumerate}
Thus, in both cases, $\CC_2(\A)$ is bi-Serre.
\end{remark}

We conclude with a natural question suggested by the preceding
classification.

\begin{question}
Which Quillen exact structures on $H^0(\A)$ lift to exact dg structures on $\A$? More generally, how can one describe the relationship between the posets
\[
\Ex_{\dg}(\A)
\qquad\text{and}\qquad
\Ex_{\mathrm Q}\bigl(H^0(\A)\bigr)?
\]
Here $\Ex_{\mathrm Q}\bigl(H^0(\A)\bigr)$ denotes the poset of Quillen exact structures on $H^0(\A)$, ordered by inclusion.
\end{question}

\medskip
\noindent
{\bf Acknowledgement.}
Y.\ O.\ is supported by JSPS KAKENHI (grant JP22K13893).

%%%%%%%%%%%%%%%%%%%%%%%%%%%%%%%%%%%%%%%%%%%%%%%%%%%%%%%%

%%%%%%%%%%%%%%%%%%%%%%%%%%%%%%%%%%%%%%%%%%%%%%%%%%%%%%%%

\end{document}